\documentclass[12pt,bezier]{article}
\usepackage{times}
\usepackage{booktabs}
\usepackage{pifont}
\usepackage{floatrow}
\usepackage{caption}
\usepackage{mathrsfs}
\usepackage[fleqn]{amsmath}
\usepackage{amsfonts,amsthm,amssymb,mathrsfs,bbding}
\usepackage{txfonts}
\usepackage{graphics,multicol}
\usepackage{float}
\usepackage{graphicx}
\usepackage{tikz}
\usepackage{color}
\usepackage{caption}
\usepackage{cite}
\usepackage{latexsym,bm}
\usepackage{indentfirst}

\evensidemargin=\oddsidemargin\topmargin=-1.5cm

\newtheorem{prob}{Problem}
\newtheorem{lem}{Lemma}[section]
\newtheorem{thm}{Theorem}[section]

\newtheorem{conj}{Conjucture}

\newtheorem{clm}{Claim}[section]

\theoremstyle{definition}

\addtocounter{section}{0}

\begin{document}
	
\title{A generalization on spectral extrema of $K_{s,t}$-minor free graphs
\footnote{Supported by National Natural Science Foundation of China (Nos.  12061074, 12171154, 11971274).}}
\author{
{Yanting Zhang$^{a}$, Zhenzhen Lou$^{a,b}$\footnote{Corresponding author.
\it Email addresses:
xjdxlzz@163.com (Z. Lou).} }\\[2mm]
\small $^a$ School of Mathematics, East China University of Science and Technology,\\
\small Shanghai, 200237, China\\
\small $^b$ College of Mathematics and System Science,
Xinjiang University, Urumqi, 830046, China }
\date{}
\maketitle{\flushleft\large\bf Abstract}
The spectral extrema problems on forbidding minors have aroused wide attention.
Very recently, Zhai and Lin [J. Combin. Theory Ser. B 157 (2022) 184--215] determined the extremal graph with maximum adjacency spectral radius
among all $K_{s,t}$-minor free graphs of sufficiently large order.
The matrix $A_{\alpha}(G)$ is a generalization of the adjacency matrix $A(G)$, which is defined by Nikiforov \cite{Nikiforov2} as
$$A_{\alpha}(G) = \alpha D(G) + (1 - \alpha)A(G),$$
where $0\leq\alpha \leq1$.
Given a graph $F$, the $A_\alpha$-spectral extrema problem is to determine the maximum spectral radius of $A_{\alpha}(G)$ or characterize the extremal graph among all graphs with no
subgraph isomorphic to $F$.
For $\alpha=0$, the matrix $A_{\alpha}(G)$ is exactly the adjacency matrix $A(G)$.
Motivated by the nice work of Zhai and Lin, in this paper we
determine the extremal graph with maximum $A_\alpha$-spectral radius among all $K_{s,t}$-minor free graphs of sufficiently large order, where $0<\alpha<1$ and
$2\leq s\leq t$.
As by-products,  we  completely solve the Conjecture  posed by Chen and Zhang  in [Linear Multilinear Algebra 69 (10) (2021) 1922--1934].

\begin{flushleft}
\textbf{Keywords:} $K_{s, t}$-minor free; local edge maximality; local degree sequence majorization; double eigenvectors transformation; $A_{\alpha}(G)$-matrix
\end{flushleft}
\textbf{AMS Classification:} 05C50 05C75

\section{Introduction}
Let $G=(V(G),E(G))$ be an undirected simple graph.
The \textit{adjacency matrix}  of $G$ is the $n \times n$ matrix   $A(G)=(a_{ij})$, where $a_{ij}=1$
if $v_{i}$ is adjacent to $v_{j}$, and $0$ otherwise.
The \textit{$Q$-matrix} (or \textit{signless Laplacian matrix}) of $G$ is defined as $Q(G)=D(G)+A(G)$,
where $D(G)$ is the diagonal matrix of vertex degrees of $G$. The largest eigenvalue of $Q(G)$,
denoted by $q(G)$, is called the \textit{$Q$-index} (or \textit{signless Laplacian spectral radius}) of $G$.
For two vertex disjoint graphs $G$ and $H$, we denote by $G \cup H$ the union of $G$ and $H$, and $G \vee H$  the join of $G$ and $H$, i.e., joining every vertex of $G$ to every vertex of $H$.
Denote by $kG$ the union of $k$ disjoint copies of $G$.
As usual, denote by $K_{n}$ a \textit{complete graph} of order $n$, and $K_{m, n}$ a \textit{complete bipartite graph} on $m + n$ vertices.
Let $F_{s, t}(n):\cong K_{s-1} \vee(p K_{t}\cup K_{r})$, where $2 \leq s \leq t, n-s+1=p t+r$ and $1\leq r\leq t$.
It is easy to check that $F_{s, t}(n)$ is a $K_{s, t}$-minor free graph of order $n$. \
For a graph $G$, let $\overline{G}$ be its complement. Denote by $S^{1}(G)$ a graph obtained from a graph $G$ by subdividing once of an edge $uv$ with minimum degree sum $d_{G}(u)+d_{G}(v)$.
Denote by $H^{\star}$ the Petersen graph.
Let $H_{s, t}:\cong(\beta-1)K_{1,s} \cup K_{1,\alpha}$, where  $1\leq s\leq t$, $\beta=\left\lfloor\frac{t+1}{s+1}\right\rfloor$ and $\alpha=t-(\beta-1)(s+1)\geq s$.
Given two graphs $G$ and $H$, $H$ is a \textit{minor} of $G$
if $H$ can be obtained from a subgraph of $G$ by contracting edges.
A graph is said to be \textit{$H$-minor free} if it does not contain $H$ as a minor.

Minors play a key role in graph theory, and extremal problems on forbidding
minors have attracted appreciable amount of interests.
Firstly, it is very useful for
studying the structures and properties of graphs. For example,
every planar graph is $\{K_{3,3}, K_5\}$-minor free and every outer-planar graph is $\{K_{2,3}, K_4\}$-minor free.
Secondly, one of the problems in extremal graph theory is to concern about the maximum
number of edges for graphs that do not contain a given $H$ as a minor. It is known that a
planar graph has at most $3n-6$ edges and an outer planar graph has at most $2n-3$ edges, see \cite{A.B}.

In 2017, Nikiforov \cite{Nikiforov2} provided a unified extension of both the adjacency spectral
radius and the signless Laplacian spectral radius. For a graph $G$, it was proposed by Nikiforov \cite{Nikiforov2}
to study the family of matrices $A_{\alpha}(G)$ defined as
$$A_{\alpha}(G) = \alpha D(G) + (1 - \alpha)A(G),$$
where $0\leq\alpha \leq1$.
The \textit{$A_\alpha$-spectral radius} (or \textbf{$\alpha$-index} ) of $G$, denoted by $\rho_{\alpha}(G)$, is the largest eigenvalue of $A_{\alpha}(G)$.
Nikiforov \cite{Nikiforov2} posed the $A_{\alpha}$-spectral extrema problem:
\begin{prob}\label{prob-0}
Given a graph $F$, what is the maximum $A_\alpha$-spectral radius of a graph $G$ of order $n$, with no
subgraph isomorphic to $F$?
\end{prob}
At present, for $0 \leq \alpha<1$, the $A_{\alpha}$-spectral extrema problem is done when $F$ is a $K_r$ (see \cite{Nikiforov2}), $F$ is a $K_r$-minor (see \cite{M.T} and \cite{C.M}), and $F$ is a star forest (see \cite{C.M2} and \cite{C.M}). In this paper, we pay our attention on Problem \ref{prob-0}
when $F$ is a $K_{s,t}$-minor. For special value $\alpha\in \left\lbrace 0, \frac{1}{2}\right\rbrace$, there are plentiful results.


For the special value $\alpha=0$, then $A_\alpha(G) = A(G)$ of $G$.
In 2007,
Nikiforov \cite{Nikiforov1} (resp. Zhai and Wang \cite{M.Q1} ) obtained a sharp upper bound of adjacency  spectral radius over all $K_{2,2}$-minor free graphs of odd (resp. even) order, and determined the extremal graph.
In 2017, Nikiforov \cite{Nikiforov}
established a sharp upper bound of adjacency  spectral radius over all $K_{2,t}$-minor free graphs of large
order $n$ for $t\ge3$, and determined the extremal graph when $t \mid n-1$. In addition, Nikiforov determined the extremal graph when $t=3$ for all $n$. In 2019,
Tait \cite{M.T}
obtained an upper bound of adjacency spectral radius
over all  $K_{s,t}$-minor free graphs of large order $n$ for  $2\leq s\leq t$, and determined  the   extremal graph when
$t \mid  n-s+1$. In 2021, Wang, Chen and Fang \cite{B.W} determined the extremal graph with maximum adjacency spectral radius
over all  $K_{3,3}$ (resp. $K_{2,4}$)-minor free graphs of large order.
In 2022, Zhai and Lin \cite{M.Q} improved  Tait's result \cite{M.T} by removing the condition $t \mid n-s+1$.
Meanwhile, they
determined the extremal graph with maximum adjacency spectral radius over all $K_{1,t}$-minor free graphs.
Thus the adjacency spectral extremal problem on $K_{s,t}$-minor free graphs
is completely solved for large order.

For the special values $\alpha=\frac{1}{2}$, then $A_\alpha(G) = Q(G)$ of $G$.
In 2022, Zhang and Lou \cite{Y.T} determined the unique extremal graph with maximum $Q$-index among all $n$-vertex connected $K_{1,t}$-minor free graphs.
For $2 \leq s \leq t$, Chen and Zhang proposed the following conjecture in \cite{C.M1}.
\begin{conj}(\cite{C.M1})\label{conj::1}
Let $2 \leq s \leq t$ and $G$ be a $K_{s,t}$-minor free graph of sufficiently large order $n$. Then
$$q(G) \leq q(F_{s,t}(n))$$
with equality if and only if $G \cong F_{s,t}(n)$.
\end{conj}
The previous results showed that the Conjecture \ref{conj::1} is true in several cases. In 2013, Freitas, Nikiforov and Patuzzi \cite {M.F} showed that $F_{2,2}(n)$ is the extremal graph with maximum $Q$-index among all $K_{2,2}$-minor free graphs for $n \geq 4$.
In 2021, Chen and Zhang \cite {C.M1} showed that $F_{2,3}(n)$ (resp. $F_{3,3}(n)$) is the extremal graph with maximum $Q$-index among all $K_{2,3}$ (resp. $K_{3,3}$)-minor free graphs for $n \geq 22$ (resp. $n \geq 1186$).
They also obtained an upper bound of $Q$-index
over all $K_{2,t}$-minor free graphs of order $n\geq t^{2} +4t+1$ with $t\ge 3$, and proved the extremal graph is $F_{2,t}(n)$ when $t \mid n-1$.

However, for general $\alpha\in(0,1)$,
the related results are few. In 2022, Chen, Liu and Zhang \cite{C.M}
obtained an upper bound of $A_\alpha$-spectral radius over all  $K_{s,t}$-minor free graphs of large order $n$ for  $2\leq s\leq t$, and proved that $F_{s,t}(n)$ is the extremal graph when
$t \mid  n-s+1$.
Therefore,
Problem \ref{prob-0} is still open when $F$ is a $K_{s,t}$-minor for $t \nmid  n-s+1$.
Naturally, we want to overcome the following problem:
\begin{prob}\label{1.1}
For $2\leq s\leq t$, what is the extremal graph with maximum $A_\alpha$-spectral radius among all $K_{s, t}$-minor free graphs of sufficiently large order for $0 < \alpha <1$?  Whether the extremal graph is consistent for $0 \leq \alpha <1$?
\end{prob}


In this paper, the \emph{$A_\alpha$-spectral extremal graph} is defined by a graph with maximum $A_\alpha$-spectral radius among all $K_{s, t}$-minor free graphs of sufficiently order $n$ for $0 < \alpha<1$ and $2\leq s\leq t$.
In order to characterize the structure of the {$A_\alpha$-spectral extremal graph,
our first challenge is to show that
the $A_\alpha$-spectral extremal graph contains a $(s-1)$-clique dominating set, the way of which is different from the adjacency spectra. In addition, we prove that the $A_\alpha$-spectral extremal graph has a property of
local edge maximality. Meanwhile, we apply  the double eigenvectors transformation technique to the $A_{\alpha}(G)$-matrix and get that
the $A_\alpha$-spectral extremal graph has a property of local degree sequence majorization.
Finally, we completely determined the $A_\alpha$-spectral extremal graph as follows.

\begin{thm}\label{thm::1.1}
Let $0<\alpha<1$, $2 \leq s \leq t$, $\beta=\left\lfloor\frac{t+1}{s+1}\right\rfloor$ and $G^{*}$ be a $K_{s,t}$-minor free graph of sufficiently large order $n$ with the maximum $A_\alpha$-spectral radius, where $n-s+1=pt+r$ and $1\leq r\leq t$. Then
$$
G^{*} \cong \begin{cases}K_{s-1} \vee \left((p-1) K_{t} \cup \overline{H^{\star}}\right) & \text { if } r=2, t=8 \text { and } \beta=1 ; \\ K_{s-1} \vee\left( (p-1) K_{t} \cup S^{1}\left(\overline{H_{s, t}}\right)\right) & \text { if } r=\beta=2; \\ K_{s-1} \vee\left((p-r) K_{t} \cup r \overline{H_{s, t}}\right) & \text { if } r \leq 2(\beta-1) \text { except } r=\beta=2 ; \\ K_{s-1} \vee\left(p K_{t} \cup K_{r}\right) & \text { otherwise. }\end{cases}
$$
\end{thm}

Taking $\alpha=\frac{1}{2}$ in Theorem \ref{thm::1.1},  we  completely solve Conjecture \ref{conj::1}.
On the other hand, we completely answer the front part of Problem \ref{1.1}. Combining the result of the extremal graph with maximum adjacency spectral radius among all $K_{s,t}$-minor free graphs in \cite{M.Q},
we give a positive answer to the last part of Problem \ref{1.1}.

Our proofs are based on the structural analysis of the $A_\alpha$-spectral extremal graph $G^{*}$, which is motivated by the work of Zhai and Lin in \cite{M.Q}. Some special notations, terminologies and lemmas will be presented in Section 2. We also prove that $G^{*}$ has a $(s-1)$-clique dominating set in Section 2. The proofs of $G^{*}$ has the structural properties of local edge maximality and local degree sequence majorization are shown in Section 3. The proof of Theorem \ref{thm::1.1} will be shown in Section 4.

\section{Preliminaries}
In this section, we will list some symbols and  useful lemmas. Let $\pi(G)=\left( d_{1}, d_{2},\dots , d_{n}\right) $ be the non-increasing degree sequence of an $n$-vertex graph $G$.
For $A, B \subseteq V(G)$, $e(A, B)$ is denoted the number of the edges of $G$ with one end vertex in $A$ and the other in $B$.
For a vertex  $v \in V(G)$, we write $N_{G}(v)$ for the set of neighbors of $v$ in $G$. Let $d_{G}(v)$ be the degree of a vertex  $v$ in $G$. Let $H$ be a subgraph of $G$, we write $N_{H}(v)$ for the set of neighbors of $v$ in $V(H)$, and  $d_{H}(v)$  for the number of neighbors of  $v$  in $V(H)$, that is,  $d_{H}(v)=\left|N_{H}(v)\right|=|N_{G}(v) \cap V(H)|$.  Let $\Delta(G)$ and $\delta(G)$ denote the maximum degree and minimum degree of $G$, respectively.
For $A \subseteq V(G)$, the graph $G[A]$ is the \textit{induced subgraph} by $A$. $G[A]$ is called a clique if it is a complete subgraph of $G$. Let $G$ be an $n$-vertex graph and $A\subseteq V(G)$, then $A$ is called a \textit{clique dominating set} of $G$, if $d_{G}(v)=n-1$ for any $v\in A$.
Let $K_{n}-e$ be a graph obtained by deleting one edge from a complete graph $K_{n}$.
For graph notations and concepts undefined here, readers are referred to \cite {A.B}.

\begin{lem}(\cite{Nikiforov2})\label{lem::2.2}
Let $0\leq\alpha<1$, then the $\alpha$-index of any proper subgraph of a connected graph is smaller than the $\alpha$-index of the original graph.
\end{lem}

The following result is from  the proof of Theorem 1.2 in \cite{C.M}.

\begin{lem}(\cite{C.M})\label{lem::2.1}
Let  $G$ be a $K_{s,t}$-minor free graph of sufficiently large order $n$ with maximum $\alpha$-index, where $2 \leq s \leq t$ and  $0 < \alpha < 1$. Then $G$ contains a vertex set $K=\left\{v_{1}, v_{2}, \ldots, v_{s-1}\right\}$ such that all of $v_{i}$ have common neighborhood of size $n-s+1$ in $G$. That is, $d_{G-K}\left(v_{i}\right)=n-s+1$ for $i \in\{1,2, \ldots, s-1\}$.
\end{lem}

\begin{lem}(Lemma 2.1, \cite{C.M})\label{lem::2.1'}
Let $0<\alpha<1, s \geq 2$, and $n \geq s-1$. If $G=K_{s-1} \vee \overline{K}_{n-s+1}$, then $\rho_\alpha(G) \geq \alpha(n-1)+(1-\alpha)(s-2)$.
\end{lem}

Recall that $F_{s, t}(n):\cong K_{s-1} \vee(p K_{t}\cup K_{r})$, where $2 \leq s \leq t, n-s+1=p t+r$ and $1\leq r\leq t$.

\begin{lem}\label{lem::2.3}
Let $0 < \alpha <1$, $2 \leq s \leq t$ and $n \geq  s-1+\frac{t^{2}-1}{\alpha}$. Then
$\rho_{\alpha}(F_{s, t}(n))$ is no more than the largest root of $g(x)=0$
and $\rho_{\alpha}(F_{s, t}(n))$ is larger than the largest root of $h(x)=0$, where
\begin{equation*}
\begin{aligned}
h(x)=&x^2-(\alpha n+s+t-3) x+
(\alpha(n-s+1)+s-2)(\alpha(s-1)+t-1)\\&-(1-\alpha)^2(s-1)(n-s),
\end{aligned}
\end{equation*}
and
$
g(x)=h(x)-(1-\alpha)^2(s-1).
$
\end{lem}

\begin{proof}
Let $\rho_{\alpha}=\rho_{\alpha}(F_{s, t}(n))$ and
$\mathbf{x}$ be a positive eigenvector of $A_{\alpha}(F_{s, t}(n))$ corresponding to $\rho_{\alpha}$. By symmetry and the Perron-Frobenius theorem, all vertices of subgraphs $K_{s-1}$, $p \cdot K_{t}$, or $K_{r}$ in $F_{s, t}(n):=K_{s-1} \vee(p \cdot K_{t} \cup K_{r} )$ have the same eigenvector components respectively, which are denoted by $x_{1}$, $x_{2}$, $x_{3}$, respectively. We consider the following two cases.
	
{\flushleft\bf Case 1. $r=t$.} By $A_{\alpha}(F_{s, t}(n)) \mathbf{x}=\rho_{\alpha} \mathbf{x}$, it is easy to see that
$$
\begin{aligned}
&(\rho_{\alpha}-\alpha(n-1)-(1-\alpha)(s-2)) x_{1} =(1-\alpha)(n-s+1) x_{2}, \\
&(\rho_{\alpha}-\alpha(s+t-2)-(1-\alpha)(t-1))x_{2} =(1-\alpha)(s-1) x_{1}.
\end{aligned}
$$
Then $\rho_{\alpha}$ is the largest root of $g(x)=0$, where
$
g(x)=h(x)-(1-\alpha)^2(s-1).
$
Since $0 <\alpha <1$ and $s \geq2$, we find that $\rho_{\alpha}$ is larger than the largest root of $h(x)=0$.

{\flushleft\bf Case 2. $1 \leq r<t$.} By $A_{\alpha}(F_{s, t}(n))\mathbf{x}=\rho_{\alpha} \mathbf{x}$, it is easy to see that
$$
\begin{aligned}
&(\rho_{\alpha}-\alpha(n-1)-(1-\alpha)(s-2)) x_{1} =(1-\alpha)(n-s-r+1)x_{2}+(1-\alpha)rx_{3}, \\
&(\rho_{\alpha}-\alpha(s+t-2)-(1-\alpha)(t-1))x_{2} =(1-\alpha)(s-1) x_{1},\\
&(\rho_{\alpha}-\alpha(s+r-2)-(1-\alpha)(r-1))x_{3}=(1-\alpha)(s-1) x_{1}.
\end{aligned}
$$
Then $\rho_{\alpha}$ is the largest root of $f(x)=0$, where
$$
f(x)=(x-\alpha(s+r-2)-(1-\alpha)(r-1)) g(x)-r(\alpha-1)^{2}(s-1)(r-t).
$$
Since $\rho_{\alpha}>\rho_{\alpha}\left(K_{s+r-1}\right)=s+r-2$, $0 < \alpha <1$ and $s\geq2$, we have
$$\rho_{\alpha}-\alpha(s+r-2)-(1-\alpha)(r-1)>(s+r-2)-\alpha(s+r-2)-(1-\alpha)(r-1)=(1-\alpha)(s-1)>0.$$
Moreover, since $1 \leq r<t$, we see that
$$
g(\rho_{\alpha})=\frac{r(\alpha-1)^{2}(s-1)(r-t)}{\rho_{\alpha}-\alpha(s+r-2)-(1-\alpha)(r-1)}<0,
$$
which implies that $\rho_{\alpha}$ is less than the the largest root of $g(x)=0$.
	
Moreover, we see that
\begin{equation*}
\begin{aligned}
f(x)=&(x-\alpha(s+r-2)-(1-\alpha)(r-1)) h(x)-
(x + 1 -\alpha (s - 1) - r)(s - 1)(\alpha-1)^2\\&-r(\alpha-1)^{2}(s - 1)(r-t).
\end{aligned}
\end{equation*}
Thus
$$
h(\rho_{\alpha})=\frac{[\rho_{\alpha}+1- \alpha (s-1)- r(1-r+t)](\alpha-1)^{2}(s-1)}{\rho_{\alpha}-\alpha(s+r-2)-(1-\alpha)(r-1)}.
$$
Notice that $G$ contains $K_{s-1} \vee \overline{K}_{n-s+1}$ as a proper subgraph. By Lemma \ref{lem::2.1'}, we have	
\begin{equation*}
\begin{aligned}
\rho_{\alpha}>&\rho_{\alpha}\left(K_{s-1} \vee \overline{K}_{n-s+1}\right)\geq \alpha(n-1)+(1-\alpha)(s-2).
\end{aligned}
\end{equation*}
Recall that $0 < \alpha <1$, $s\geq2$, $1\leq r< t$ and $n \geq  s-1+\frac{t^{2}-1}{\alpha}$. We obtain
\begin{equation*}
\begin{aligned}
\rho_{\alpha}+1-\alpha (s-1)-r(1-r+t)>&\alpha(n-1)+(1-\alpha)(s-2)+1-\alpha (s-1)-(t-1)(1+t)\\
=&\alpha (n-1) + s-1-\alpha(2s-3)-t^{2}+1\\
\geq&\alpha (n-1) -\alpha(s-2)-t^{2}+1\\
\geq&\alpha \left( \left( s-1+\frac{t^{2}-1}{\alpha}\right) -1\right)  -\alpha(s-2)-t^{2}+1\\
=&0.
\end{aligned}
\end{equation*}
Hence, $h(\rho_{\alpha})>0$. Furthermore, since
\begin{equation*}
\begin{aligned}
n&\geq s-1+\frac{t^{2}-1}{\alpha}=2(s-1)+\frac{t^{2}-1-\alpha(s-1)}{\alpha}\geq2(s-1)+\frac{t^{2}-1-(s-1)}{\alpha}\\
&\geq2(s-1)+\frac{t-(s-1)}{\alpha}=2(s-1)+\frac{t-s+1}{\alpha},
\end{aligned}
\end{equation*}
we have
\begin{equation*}
\begin{aligned}
\rho_{\alpha}>&\alpha(n-1)+(1-\alpha)(s-2)
=\frac{\alpha n}{2} +\frac{\alpha n}{2}-\alpha (s-1)+s-2\\
\geq&\frac{\alpha n}{2} +\frac{\alpha}{2}\left(2(s-1)+\frac{t-s+1}{\alpha} \right) -\alpha (s-1)+s-2=\dfrac{\alpha n+s+t-3}{2}.
\end{aligned}
\end{equation*}
It follows that $\rho_{\alpha}$ is larger than the largest root of $h(x)=0$.
\end{proof}

\begin{lem}\label{lem::2.5}
Let $0<\alpha<1$, $s\geq4$, $n\geq s$ and $G\cong (K_{s-1}-e) \vee \overline{K}_{n-s+1}$. Then $\rho_{\alpha}(G)$ is the largest root of $f(x) = 0$, where
\begin{equation*}
\begin{aligned}
f(x)=&x^3 -(2\alpha n + s - 4)x^2 +(\alpha^2n^2 +3\alpha ns - \alpha s^2 - 6\alpha n +\alpha s- ns + s^2 - 2\alpha + n -
4s +\\& 7)x -2\alpha^2 n^2 s + \alpha^2 ns^2 + 2 \alpha^2 n^2 - \alpha^2 n s + \alpha n^2 s - \alpha ns^2 + 2\alpha^2 n - \alpha n^2 + 6\alpha ns -
2\alpha s^2 \\&-9\alpha n + 4\alpha s -2ns + 2s^2 - 2\alpha + 4n - 6s + 4.
\end{aligned}
\end{equation*}	
Moreover, $\rho_{\alpha}(G) > \alpha(n-1)+(1-\alpha)(s-4).$
\end{lem}

\begin{proof}
Firstly, denote by $\{v_{1}, v_{2}, \cdots, v_{s-3}, w_{1}, w_{2}\}$ the  vertex set of $K_{s-1}-e$ in the representation $G:\cong (K_{s-1}-e) \vee \overline{K}_{n-s+1}$, where $d_{G}(v_{i})=n-1$ for $i\in \{1,2,\dots,s-3\}$ and $d_{G}(w_{1})=d_{G}(w_{2})=n-2$. Set for short $\rho_{\alpha}=\rho_{\alpha}(G)$ and let $\mathbf{x}=\left(x_{v}\right)_{v \in V(G)}$ be the Perron vector of $A_{\alpha}(G)$ with respect to $\rho_{\alpha}$. By symmetry, we have $x_{v_{1}}=\cdots=x_{v_{s-3}}$ and  $x_{w_{1}}=x_{w_{2}}$. Additionally, $x_{z}=x_{z_{1}}$ for any two vertices
$z, z_1\in V(G)\backslash \{v_{1}, v_{2}, \cdots, v_{s-3}, w_{1}, w_{2}\}$. By eigen-equations of $A_{\alpha}(G)$ on $v_{1}$, $w_{1}$ and $z_{1}$, we have
$$
\begin{aligned}
&(\rho_{\alpha}-\alpha(n-1)-(1-\alpha)(s-4)) x_{v_{1}} =2(1-\alpha)x_{w_{1}}+(1-\alpha)(n-s+1)x_{z_{1}}, \\
&(\rho_{\alpha}-\alpha(n-2))x_{w_{1}} =(1-\alpha)(s-3) x_{v_{1}}+(1-\alpha)(n-s+1)x_{z_{1}},\\
&(\rho_{\alpha}-\alpha(s-1))x_{z_{1}}=(1-\alpha)(s-3) x_{v_{1}}+2(1-\alpha)x_{w_{1}}.
\end{aligned}
$$
Then $\rho_{\alpha}$ is the largest real root of $f(x) = 0$, where
\begin{equation*}
\begin{aligned}
f(x)=&x^3 -(2\alpha n + s - 4)x^2 +(\alpha^2n^2 + 3\alpha ns - \alpha s^2 - 6\alpha n +\alpha s - ns + s^2
- 2\alpha + n -  4s \\
&+ 7)x -2\alpha^2 n^2 s + \alpha^2 ns^2 + 2 \alpha^2 n^2 - \alpha^2 n s + \alpha n^2 s - \alpha ns^2 + 2\alpha^2 n - \alpha n^2 + 6\alpha ns -\\
&2\alpha s^2 - 9\alpha n + 4\alpha s -2ns + 2s^2 - 2\alpha + 4n - 6s + 4.
\end{aligned}
\end{equation*}
Since $0 < \alpha < 1$, $s\geq4$ and $n \geq s$, we have
\begin{equation*}
\begin{aligned}
(\alpha(s-3)-s)(n-s+3) + 8\leq&((s-3)-s)(n-s+3)  + 8
=-3(n-s+3) + 8\\
\leq&-3(s-s+3) + 8
<0,
\end{aligned}
\end{equation*}
and so
$f(\alpha(n-1)+(1-\alpha)(s-4))=((\alpha(s-3)-s)(n-s+3) + 8)(s - 3)(\alpha - 1)^2<0,$
which implies that $\rho_{\alpha} >\alpha(n-1)+(1-\alpha)(s-4).$

This completes the proof.
\end{proof}

\begin{lem}\label{lem::2.6}
Let $0< \alpha <1$, $t\geq 2$, $s\geq3$, $C\geq\frac{1}{\alpha}$ and
$$n \geq \max\{ 2s-3+\frac{t-s+4}{\alpha}, \frac{(1-\alpha)(s-1)(C(s+t)+2)}{2}+s+t \}.$$ Suppose that $H$ is a graph of order $n-s+1$ and $G\cong (K_{s-1}-e) \vee H$, particularly, defined by $G^{\prime}=G$ when $H$ is $(t-1)$-regular. If $\Delta(H) \leq t-1$, then $\rho_{\alpha}(G) \leq \rho_{\alpha}(G^{\prime})$, eauality holds if and only if $G\cong G^{\prime}$,
moreover, $\rho_{\alpha}(G^{\prime})$ is less than the largest root of $h(x)=0$, where
\begin{equation*}
\begin{aligned}
h(x)=&x^2-(\alpha n+s+t-3) x+
(\alpha(n-s+1)+s-2)(\alpha(s-1)+t-1)\\&-(1-\alpha)^2(s-1)(n-s).
\end{aligned}
\end{equation*}
\end{lem}

\begin{proof}
Denote by $\{v_{1}, v_{2}, \cdots, v_{s-3}, w_{1}, w_{2}\}$ the  vertex set of $K_{s-1}-e$ in the representation $G:\cong (K_{s-1}-e) \vee H$, where $d_{G}(v_{i})=n-1$ for $i\in \{1,2,\dots,s-3\}$ and $d_{G}(w_{1})=d_{G}(w_{2})=n-2$. Set for short $\rho_{\alpha}=\rho_{\alpha}(G)$ and let
$\mathbf{x}=\left(x_{v}\right)_{v \in V(G)}$ be the Perron vector of $A_{\alpha}(G)$ with respect to $\rho_{\alpha}$. Clearly, $x_{v_{1}}=\cdots=x_{v_{s-3}}$ and $x_{w_{1}}=x_{w_{2}}$ by the symmetry. Choose a vertex $z_{1} \in V(H)$ such that
$$x_{z_{1}}=\max _{v \in V(H)} x_{v}.$$
Since $\Delta(H) \leq t-1$, we have $d_{G}(z_{1})=d_{H}(z_{1})+s-1\leq s+t-2$.
	
{\flushleft\bf Case 1. $s=3$.}
Then
$
h(x)=x^2-(\alpha n+t) x+(\alpha(n-2)+1)(2\alpha+t-1)-2(1-\alpha)^2(n-3)$.
By eigenequations of $A_{\alpha}(G)$ on $w_{1}$ and $z_{1}$, we have
$$
\begin{aligned}
\rho_{\alpha}x_{w_{1}}&=\alpha(n-2)x_{w_{1}}+(1-\alpha)\sum_{v \in V(H)} x_{v}\leq\alpha(n-2)x_{w_{1}}+(1-\alpha)(n-2)x_{z_{1}},\\
\rho_{\alpha}x_{z_{1}}&=\alpha d_{G}(z_{1})x_{z_{1}}+2(1-\alpha)x_{w_{1}}+(1-\alpha)\sum_{v \in N_{H}(z_{1})} x_{v}
\\&\leq ((t+1)\alpha+(1-\alpha)(t-1) )x_{z_{1}}+2(1-\alpha)x_{w_{1}},
\end{aligned}
$$
and thus
\begin{equation}\label{equ::2_}
(\rho_{\alpha}-\alpha(n-2)) x_{w_{1}} \leq (1-\alpha)(n-2)x_{z_{1}},
\end{equation}
\begin{equation}\label{equ::3_}
(\rho_{\alpha}-(t+1)\alpha-(1-\alpha)(t-1) ) x_{z_{1}} \leq 2(1-\alpha)x_{w_{1}}.
\end{equation}
Notice that $n \geq 2s-3+\frac{t-s+4}{\alpha}=3+\frac{t+1}{\alpha}$. Then we obtain
\begin{equation}\label{equ::2-}
\begin{aligned}
\rho_{\alpha}&\geq\rho_{\alpha}(K_{2,n-2})=\dfrac{\alpha n+\sqrt{(\alpha n)^2+8(n-2)(1-2\alpha)}}{2}
=\dfrac{\alpha n+\sqrt{(\alpha (n-4))^2+8(\alpha - 1)^2 (n - 2)}}{2}\\&
>\dfrac{\alpha n+\alpha (n-4)}{2}
=\alpha (n-2)
\geq\alpha \left( \left( 3+\frac{t+1}{\alpha}\right) -2\right)
>(t+1)\alpha+(1-\alpha)(t-1).
\end{aligned}
\end{equation}
Now, we can multiply (\ref{equ::2_})-(\ref{equ::3_}), obtaining
$$
\rho_{\alpha}^2 -(\alpha n + t - 1)\rho_{\alpha} +\alpha nt + 3 \alpha n - 2 \alpha t - 6\alpha - 2n + 4 \leq 0.
$$
This implies that $\rho_{\alpha}$ is no more than the largest root of $g(x)=0$, where
$$g(x)=x^2 -(\alpha n + t - 1)x +\alpha nt + 3 \alpha n - 2 \alpha t - 6\alpha - 2n + 4 .$$
If $\rho_{\alpha}$ is equal to the largest real root of $g(x)=0$, then equalities in (\ref{equ::2_})-(\ref{equ::3_}) hold. Therefore, for any vertex $z \in V(H)$, we have $x_{z}=x_{z_{1}}$ and
$$
\begin{aligned}
\rho_{\alpha}-\alpha d_{G}(z)x_{z}
&=2(1-\alpha)x_{w_{1}}+(1-\alpha)\sum_{v \in N_{H}(z)} x_{v}\\
&\leq 2(1-\alpha)x_{w_{1}}+(1-\alpha)(t-1)x_{z_{1}}=\rho_{\alpha}-\alpha d_{G}(z_{1})x_{z_{1}},
\end{aligned}
$$
which implies $d_{G}(z)=s+t-2$, that is, $d_{H}(z)=t-1$. Hence, $H$ is a $(t-1)$-regular graph, and so $G\cong G^{\prime}$.
Furthermore, we see that $g(x)=h(x)+x-(t+1-2\alpha(1-\alpha))$. By (\ref{equ::2-}), we have
$\rho_{\alpha}>\alpha \left( \left( 3+\frac{t+1}{\alpha}\right) -2\right)=t+1+\alpha.$
Thus,
$h(\rho_{\alpha})=-(\rho_{\alpha}-(t+1-2\alpha(1-\alpha)))\leq-(\rho_{\alpha}-(t+1))<0$.
Therefore, $\rho_{\alpha}$ is less than the largest root of $h(x)=0$.
	
{\flushleft\bf Case 2. $s\geq4$.}
By eigenequations of $A_{\alpha}(G)$ on $v_{1}$, $w_{1}$ and $z_{1}$, we have
\begin{equation*}\begin{array}{ll}
(\rho_{\alpha}-\alpha(n-1)-(1-\alpha)(s-4))x_{v_{1}}
&=2(1-\alpha)x_{w_{1}}+(1-\alpha)\sum_{v \in V(H)} x_{v}\\
&\leq2(1-\alpha)x_{w_{1}}+(1-\alpha)(n-s+1)x_{z_{1}},
\end{array}
\end{equation*}
\begin{equation*}\begin{array}{ll}
(\rho_{\alpha}-\alpha(n-2))x_{w_{1}}
&=(1-\alpha)(s-3) x_{v_{1}}+(1-\alpha)\sum_{v \in V(H)} x_{v}\\
&\leq(1-\alpha)(s-3) x_{v_{1}}+(1-\alpha)(n-s+1)x_{z_{1}},
\end{array}
\end{equation*}
\begin{equation*}\begin{array}{ll}
(\rho_{\alpha}-\alpha(s+t-2))x_{z_{1}}
&\leq (\rho_{\alpha}-\alpha d_{G}(z_{1})x_{z_{1}}\\
&=(1-\alpha)(s-3) x_{v_{1}}+2(1-\alpha)x_{w_{1}}+(1-\alpha)\sum_{v \in N_{H}(z_{1})} x_{v}\\
&\leq(1-\alpha)(s-3) x_{v_{1}}+2(1-\alpha)x_{w_{1}}+(1-\alpha)(t-1)x_{z_{1}},
\end{array}
\end{equation*}

that is,
\begin{equation}\label{equ::1}
(\rho_{\alpha}-\alpha(n-1)-(1-\alpha)(s-4)) x_{v_{1}} \leq2(1-\alpha)x_{w_{1}}+(1-\alpha)(n-s+1)x_{z_{1}},
\end{equation}
\begin{equation}\label{equ::2}
(\rho_{\alpha}-\alpha(n-2)) x_{w_{1}} \leq (1-\alpha)(s-3) x_{v_{1}}+(1-\alpha)(n-s+1)x_{z_{1}},
\end{equation}
\begin{equation}\label{equ::3}
(\rho_{\alpha}-\alpha(s+t-2)-(1-\alpha)(t-1)) x_{z_{1}} \leq (1-\alpha)(s-3) x_{v_{1}}+2(1-\alpha)x_{w_{1}}.
\end{equation}
Notice that and $G$ contains $(K_{s-1}-e) \vee \overline{K}_{n-s+1}$ as a subgraph. By Lemma \ref{lem::2.5}, we have
$$
\rho_{\alpha}\geq \rho_{\alpha}\left((K_{s-1}-e) \vee \overline{K}_{n-s+1} \right)>\alpha(n-1)+(1-\alpha)(s-4).
$$
Since $0<\alpha<1$ and $s\geq4$, we have
\begin{equation}\label{equ::4-}
\begin{aligned}
\rho_{\alpha}\geq \alpha(n-1)+(1-\alpha)(s-4)>\alpha(n-2).
\end{aligned}
\end{equation}
Recall that $n \geq  2s-3+\frac{t-s+4}{\alpha}$. Hence, we get that
\begin{equation}\label{equ::4--}
\begin{aligned}
\rho_{\alpha}\geq \alpha(n-1)+(1-\alpha)(s-4)>\alpha(s+t-2)+(1-\alpha)(t-1).
\end{aligned}
\end{equation}
Let $A=1$; $B=-(2\alpha n + s + t - 5)$; $C=a^2 n^2 + 3\alpha ns + 2\alpha nt - \alpha s^2 - \alpha st - 8\alpha n + 2\alpha s + \alpha t - ns + s^2 + st - 3\alpha + n - 5s - 4t + 11$;
$D= 10 + 4n - 6t - 8s - 6\alpha + 2s^2 + 7\alpha ns - 2\alpha^2 ns - 2\alpha^2 n^2 s + 4\alpha nt - 2\alpha st -\alpha n^2+ 4\alpha t + 3\alpha^2 n + 3\alpha^2 n^2 - 2\alpha s^2 + 6\alpha s - 2ns - 13\alpha n + 2st + \alpha n^2s + \alpha^2 ns^2 - \alpha ns^2 - \alpha^2 n^2 t + \alpha^2 nst - \alpha nst - \alpha^2 nt$.
Now, we can multiply (\ref{equ::1})-(\ref{equ::3}), obtaining
$A\rho_{\alpha}^{3}+B\rho_{\alpha}^{2}+C\rho_{\alpha}+D \leq 0$.
This implies that $\rho_{\alpha}$ is no more than the largest  root of $f(x)=0$, where
$$
f(x)=Ax^{3}+Bx^{2}+Cx+D.
$$
If $\rho_{\alpha}$ is equal to the largest real root of $f(x)=0$, then equalities   in (\ref{equ::1})-(\ref{equ::3}) hold. Therefore, for any vertex $z \in V(H)$, we have $x_{z}=x_{z_{1}}$ and
$$
\begin{aligned}
\rho_{\alpha}-\alpha d_{G}(z)x_{z}
&=(1-\alpha)(s-3) x_{v_{1}}+2(1-\alpha)x_{w_{1}}+(1-\alpha)\sum_{v \in N_{H}(z)} x_{v}\\
&\leq(1-\alpha)(s-3) x_{v_{1}}+2(1-\alpha)x_{w_{1}}+(1-\alpha)(t-1)x_{z_{1}}=\rho_{\alpha}-\alpha d_{G}(z_{1})x_{z_{1}},
\end{aligned}
$$
which implies   $d_{G}(z)=s+t-2$, that is, $d_{H}(z)=t-1$. Hence, $H$ is a $(t-1)$-regular graph, and so $G\cong G^{\prime}$. Moreover, let
\begin{equation*}
\begin{aligned}
g_{1}(x)=&x^2-(\alpha n+s+t-3) x+ (\alpha(n-s+1)+s-2)(\alpha(s-1)+t-1)\\&-(1-\alpha)^2(s-1)(n-s+1).
\end{aligned}
\end{equation*}
we find that
$$
f(x)=(x-(\alpha n-2))g_{1}(x)+2(1-\alpha)(x-((\alpha-1) n  + s + t - 2)).
$$
Note that $0<\alpha<1$. By (\ref{equ::4-}) we have $\rho_{\alpha}>\alpha(n-2)\geq\alpha n-2,$
and by (\ref{equ::4--}) we get that
$$
\begin{aligned}
\rho_{\alpha}-((\alpha-1) n  + s + t - 2)
&>\alpha(s+t-2)+(1-\alpha)(t-1)-((\alpha-1) n  + s + t - 2)\\
&=(1-\alpha)(n-s+1)>0.
\end{aligned}
$$
It follows that
$$
h(\rho_{\alpha})=-\dfrac{2(1-\alpha)(\rho_{\alpha}-((\alpha-1) n  + s + t - 2))}{\rho_{\alpha}-(\alpha n-2)}<0,
$$
Therefore, $\rho_{\alpha}$ is less than the largest root of $g_{1}(x)=0$. That is,
$
\rho_{\alpha}<\dfrac{\alpha n+s+t-3+\sqrt{R}}{2},
$
where
$$
\begin{aligned}
R=&(\alpha n+s+t-3)^{2}-4\left( (\alpha(n-s+1)+s-2)(\alpha(s-1)+t-1)-(1-\alpha)^2(s-1)(n-s+1)\right)\\
=&(\alpha n +(2C-1)(s + t) + 3)^2-4(C(C-1)(s^2 + 2st+t^2)+(C\alpha-1)ns + C\alpha nt  +\alpha ns +\\
&3s(C-1) +(3C-2)t + \alpha (2s + t)  +(1-\alpha)(t+s)s + n +(3-\alpha))\\
<&(\alpha n +(2C-1)(s + t) + 3)^2,
\end{aligned}
$$
since $0<\alpha<1$ and $C\geq\frac{1}{\alpha}$.
Hence, we have
$$
\begin{aligned}
\rho_{\alpha}<\dfrac{\alpha n+s+t-3+\alpha n +(2C-1)(s + t) + 3}{2}=\alpha n+C(s+t).
\end{aligned}
$$
Now we see that
$$
f(x)=(x-(\alpha n-2))h(x)+2(1-\alpha)(x-((\alpha-1) n  + s + t - 2))-(x-(\alpha n-2))(1-\alpha)^{2}(s-1).
$$
Note that $n\geq\frac{(1-\alpha)(s-1)(C(s+t)+2)}{2}+s+t$. Then we obtain
$$
\begin{aligned}
&2(1-\alpha)(\rho_{\alpha}-((\alpha-1) n  + s + t - 2))-(\rho_{\alpha}-(\alpha n-2))(1-\alpha)^{2}(s-1)\\
=&(1-\alpha)\left( 2(n-s-t)+(\rho_{\alpha}-(\alpha n-2))\left( -(1-\alpha)(s-1)+2\right) \right) \\
>&(1-\alpha)\left( 2(n-s-t)-(1-\alpha)(s-1)(\rho_{\alpha}-(\alpha n-2)) \right) \\
\geq&(1-\alpha)\left( 2(n-s-t)-(1-\alpha)(s-1)((\alpha n+C(s+t))-(\alpha n-2)) \right) \\
=&(1-\alpha)\left( 2(n-s-t)-(1-\alpha)(s-1)(C(s+t)+2) \right) \\
\geq&(1-\alpha)\left( 2\left( \frac{(1-\alpha)(s-1)(C(s+t)+2)}{2}+s+t-s-t\right) -(1-\alpha)(s-1)(C(s+t)+2) \right) \\
=&0.
\end{aligned}
$$
It follows that
$$
h(\rho_{\alpha})=-\dfrac{2(1-\alpha)(x-((\alpha-1) n  + s + t - 2))-(x-(\alpha n-2))(1-\alpha)^{2}(s-1)}{\rho_{\alpha}-(\alpha n-2)}<0,
$$
Therefore, $\rho_{\alpha}$ is less than the largest root of $h(x)=0$.
\end{proof}

\begin{lem}\label{lem::2.7}
Let $0<\alpha <1$, $2 \leq s \leq t$ and $G$ be a $K_{s,t}$-minor free graph of sufficiently large order $n$ with maximum $A_\alpha$-spectral radius. Then $G$ contains a vertex set $K=\left\{v_{1}, v_{2}, \ldots, v_{s-1}\right\}$ such that $d_{G}\left(v_{i}\right)=n-1$ for $i \in\{1,2, \ldots, s-1\}$.
\end{lem}

\begin{proof}
Let
\begin{equation*}
\begin{aligned}
h(x)=&x^2-(\alpha n+s+t-3) x+
(\alpha(n-s+1)+s-2)(\alpha(s-1)+t-1)\\&-(1-\alpha)^2(s-1)(n-s).
\end{aligned}
\end{equation*}
Note that $F_{s, t}(n)$ is a $K_{s, t}$-minor free graph. Then we have $\rho_{\alpha}(G) \geq \rho_{\alpha}(F_{s, t}(n))$. Furthermore, by Lemma \ref{lem::2.3}, we get that $\rho_{\alpha}(G)$ is larger than the largest root of $h(x)=0$. By Lemma \ref{lem::2.1}, $G$ contains a vertex set $K$ of size $s-1$ such that $d_{G-K}\left(v\right)=n-s+1$ for any vertex $v \in K$. Now, we need to show that $K$ induces a clique. Otherwise, we have $s\geq3$ and $G[K] \subseteq K_{s-1}-e$. Let $H=G-K$, i.e., $G=G[K] \vee H$. Since $G$ is a $K_{s, t}$-minor free graph, we have $\Delta(H)\leq t-1$. If $G[K] \cong K_{s-1}-e$, by Lemma \ref{lem::2.6}, we obtain $\rho_{\alpha}(G)$ is less than the largest root of $h(x)=0$, a contradiction. Therefore, $G[K]$ is a proper subgraph of $K_{s-1}-e$. Let $G^{\prime}$ be the graph obtained from $G$
by adding edges to $G[K]$ to make it a graph $K_{s-1}-e$. Then we have $\rho_{\alpha}(G)<\rho_{\alpha}(G^{\prime})$ by Lemma \ref{lem::2.2}. However, from Lemma \ref{lem::2.6}, $\rho_{\alpha}(G^{\prime})$ is
less than the largest root of $h(x)=0$, which is also a contradiction.
\end{proof}

\begin{lem}(\cite{D.L})\label{lem::2.4}
Let $t \geq 3$ and $n \geq t+2$. If $G$ is an $n$-vertex connected graph with no $K_{1, t}$-minor, then $e(G) \leq \binom{t}{2} +n-t$, and this is best possible for all $n, t$.
\end{lem}

For $2\leq s\leq t$, we say a graph $G$ has the $(s, t)$-property, if $G$ is $K_{a, b}$-minor free for any two positive integers $a, b$ with $a+b=t+1$ and $a \leq \min \left\{s,\left\lfloor\frac{t+1}{2}\right\rfloor\right\}$.
The following Lemma gives a equivalent
condition whether a graph $G$ has  $K_{s, t}$-minor or not.

\begin{lem}(Lemma 2.3, \cite{M.Q})\label{lem::2.10}
Let $2\leq s\leq t$ and $G$ be a graph with a clique dominating set $K$ of size $s-1$. Then $G$ is $K_{s, t}$-minor free if and only if $G-K$ has the $(s, t)$-property.
\end{lem}

Recall that $H_{s, t}:\cong(\beta-1)K_{1,s} \cup K_{1,\alpha}$, where $1\leq s\leq t$, $\beta=\left\lfloor\frac{t+1}{s+1}\right\rfloor$ and $\alpha=t-(\beta-1)(s+1)\geq s$,
$S^{1}\left(\overline{H_{s, t}}\right)$ is a graph obtained from a graph $\overline{H_{s, t}}$ by subdividing once of an edge $uv$ with minimum degree sum $d_{G}(u)+d_{G}(v)$ and  $H^{\star}$ is the Petersen graph.
The following result is from Theorem 3.1 and the proofs of Claims $3.8$-$3.9$ in \cite{M.Q}.
\begin{lem}(\cite{M.Q})\label{lem::2.16}
Let $2\leq s\leq t$ and $t\geq4$. Then
	
(i) $\overline{H_{s,t}}$ and $S^{1}\left(\overline{H_{s, t}}\right)$ have the $(s, t)$-property. Moreover, if $\left\lfloor\frac{t+1}{s+1}\right\rfloor=2$, then $\pi\left(S^{1}\left(\overline{H_{s, t}}\right)\right)=(t-1, \ldots, t-1, t-s, s+1,2)$.

(ii) $\overline{H^{\star}}$ has the $(s, t)$-property for $t=8$.
\end{lem}
\begin{lem}(Lemma 3.1, \cite{M.Q})\label{lem::2.11}
Let $2\leq s\leq t$, $t\geq4$, $\gamma = \min \left\{s,\left\lfloor\frac{t+1}{2}\right\rfloor\right\}$ and $G$ be a connected graph with $|G|=t+1$. Then $G$ has the $(s, t)$-property if and only if each component of $\overline{G}$ has at least $\gamma+1$ vertices.
\end{lem}

The following result comes from the proof of Lemma 3.2 in \cite{M.Q}.
\begin{lem}(\cite{M.Q})\label{lem::2.12}
Let $2\leq s\leq t$, $t\geq4$, $\beta=\left\lfloor\frac{t+1}{s+1}\right\rfloor$ and $G$ be a connected graph with $|G|=t+1$. If $G$ is an edge-maximal graph with the $(s, t)$-property, then $e(G)=\binom{t}{2}+\beta-1$ and $\overline{G}$ is a forest with $\beta$ components.
\end{lem}

\begin{lem}(Lemma 3.3, \cite{M.Q})\label{lem::2.13}
Let $G$ be a graph with $vw \in E(G)$ and $uw \notin E(G)$. If $d_{G}(u) \geq d_{G}(v)$, then $\pi(G) \prec \pi(G-\{v w\}+\{u w\})$ and $\pi(G) \neq \pi(G-\{v w\}+\{u w\})$.
\end{lem}

\begin{lem}(Lemma 3.4, \cite{M.Q})\label{lem::2.15}
Let $2\leq s\leq t$, $t\geq4$, $\left\lfloor\frac{t+1}{s+1}\right\rfloor \leq 2$ and $G$ be a connected graph with $|G|=t+2$. If $G$ has the $(s, t)$-property, then $e(G) \leq \binom{t}{2} +2$, and if equality holds, then $\overline{G}$ is isomorphic to either the Petersen graph $H^{\star}$ or some $H_{a, b, c}$ (see Fig. \ref{fig-1}), where $a+b+c=t-1$.
\end{lem}

\begin{figure}[hbtp]
\centering
\begin{tikzpicture}[x=0.9mm, y=0.9mm, inner xsep=0pt, inner ysep=0pt, outer xsep=0pt, outer ysep=0pt]
\path[line width=0mm] (25.74,33.08) rectangle +(155.86,74.28);
\definecolor{L}{rgb}{0,0,0}
\path[line width=0.30mm, draw=L] (60.31,92.82) circle (12.55mm);
\path[line width=0.30mm, draw=L] (80.50,57.74) circle (12.55mm);
\path[line width=0.30mm, draw=L] (40.29,56.89) circle (12.55mm);
\definecolor{F}{rgb}{0,0,0}
\path[line width=0.30mm, draw=L, fill=F] (57.23,71.77) circle (1.00mm);
\path[line width=0.30mm, draw=L, fill=F] (63.22,71.09) circle (1.00mm);
\path[line width=0.30mm, draw=L, fill=F] (59.28,66.98) circle (1.00mm);
\path[line width=0.30mm, draw=L] (63.22,71.77) -- (69.03,85.97);
\path[line width=0.30mm, draw=L] (63.05,71.77) -- (62.02,82.89);
\path[line width=0.30mm, draw=L] (63.22,71.77) -- (79.98,68.52);
\path[line width=0.30mm, draw=L] (63.05,71.60) -- (72.11,63.90);
\path[line width=0.30mm, draw=L] (57.91,71.94) -- (59.97,83.06);
\path[line width=0.30mm, draw=L] (58.08,72.11) -- (52.78,85.63);
\path[line width=0.30mm, draw=L] (57.57,71.94) -- (40.98,67.67);
\path[line width=0.30mm, draw=L] (57.57,72.29) -- (47.14,64.07);
\path[line width=0.30mm, draw=L] (59.45,67.32) -- (49.02,62.02);
\path[line width=0.30mm, draw=L] (59.28,66.64) -- (50.39,53.29);
\path[line width=0.30mm, draw=L] (59.28,66.98) -- (71.60,62.88);
\path[line width=0.30mm, draw=L] (58.94,66.47) -- (71.26,55.52);
\draw(56.80,91.87) node[anchor=base west]{\fontsize{14.23}{17.07}\selectfont $\overline{K_a}$};
\draw(35.94,54.28) node[anchor=base west]{\fontsize{14.23}{17.07}\selectfont $\overline{K_b}$};
\draw(76.88,54.28) node[anchor=base west]{\fontsize{14.23}{17.07}\selectfont $\overline{K_c}$};
\draw(52.05,72.11) node[anchor=base west]{\fontsize{14.23}{17.07}\selectfont $w$};
\draw(65.94,72.88) node[anchor=base west]{\fontsize{14.23}{17.07}\selectfont $u_1$};
\draw(57.23,61.03) node[anchor=base west]{\fontsize{14.23}{17.07}\selectfont $u_2$};
\draw(54.10,36.61) node[anchor=base west]{\fontsize{14.23}{17.07}\selectfont $H_{a,b,c}$};
\path[line width=0.30mm, draw=L] (57.49,71.88) -- (63.31,71.20);
\path[line width=0.30mm, draw=L] (126.51,57.23) circle (12.55mm);
\path[line width=0.30mm, draw=L] (146.18,92.64) circle (12.55mm);
\path[line width=0.30mm, draw=L] (167.05,57.40) circle (12.55mm);
\path[line width=0.30mm, draw=L, fill=F] (145.84,71.94) circle (1.00mm);
\path[line width=0.30mm, draw=L, fill=F] (150.29,66.64) circle (1.00mm);
\path[line width=0.30mm, draw=L, fill=F] (142.59,66.64) circle (1.00mm);
\path[line width=0.30mm, draw=L] (136.26,90.08) -- (124.97,66.81);
\path[line width=0.30mm, draw=L] (139.00,85.80) -- (130.10,67.15);
\path[line width=0.30mm, draw=L] (155.59,89.74) -- (169.28,67.32);
\path[line width=0.30mm, draw=L] (153.71,85.63) -- (165.00,67.15);
\path[line width=0.30mm, draw=L] (130.44,50.56) -- (162.26,50.21);
\path[line width=0.30mm, draw=L] (133.52,54.83) -- (158.67,55.01);
\path[line width=0.30mm, draw=L] (145.84,72.29) -- (142.08,84.43);
\path[line width=0.30mm, draw=L] (146.01,72.46) -- (148.75,84.43);
\path[line width=0.30mm, draw=L] (142.25,66.47) -- (131.81,64.42);
\path[line width=0.30mm, draw=L] (142.42,65.78) -- (134.89,58.08);
\path[line width=0.30mm, draw=L] (150.80,66.98) -- (161.24,65.10);
\path[line width=0.30mm, draw=L] (150.80,66.64) -- (159.01,59.80);
\path[line width=0.30mm, draw=L] (146.01,72.11) -- (142.59,67.15);
\path[line width=0.30mm, draw=L] (146.01,72.11) -- (150.29,66.98);
\draw(142.57,91.87) node[anchor=base west]{\fontsize{14.23}{17.07}\selectfont $K_a$};
\draw(121.06,54.28) node[anchor=base west]{\fontsize{14.23}{17.07}\selectfont $K_b$};
\draw(164.24,54.28) node[anchor=base west]{\fontsize{14.23}{17.07}\selectfont $K_c$};
\draw(139.22,70.29) node[anchor=base west]{\fontsize{14.23}{17.07}\selectfont $u_2$};
\draw(141.94,61.82) node[anchor=base west]{\fontsize{14.23}{17.07}\selectfont $u_1$};
\path[line width=0.30mm, draw=L] (142.42,65.78) -- (134.89,58.08);
\draw(150.46,68.01) node[anchor=base west]{\fontsize{14.23}{17.07}\selectfont $w$};
\draw(141.56,36.61) node[anchor=base west]{\fontsize{14.23}{17.07}\selectfont $\overline{H_{a,b,c}}$};
\end{tikzpicture}%
\caption{The graph $H_{a, b, c}$ and its complement, where $a+b+c=t-1$.}\label{fig-1}
\end{figure}
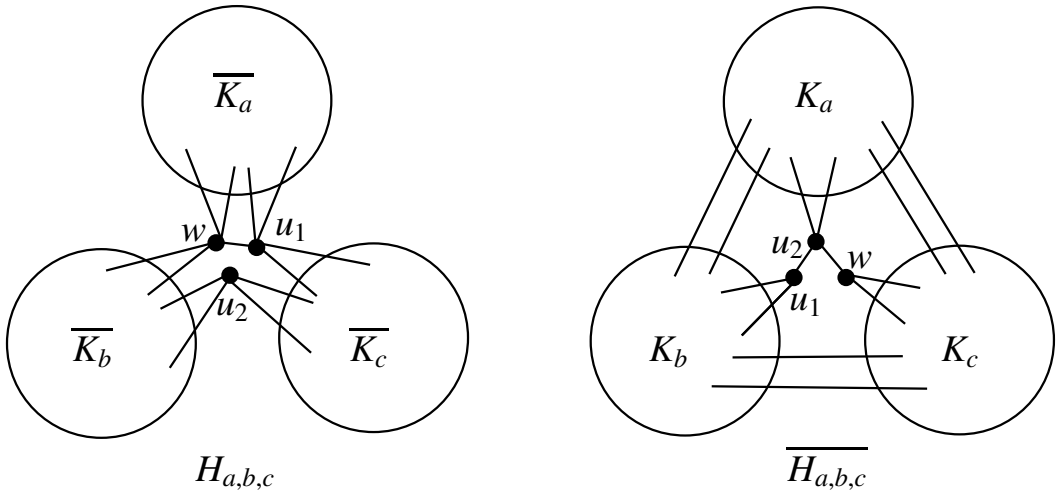

\section{Two properties of the $A_\alpha$-spectral extremal graph}
Recall that the \emph{$A_\alpha$-spectral extremal graph} is a graph with maximum $A_\alpha$-spectral radius among all $K_{s, t}$-minor free graphs of sufficiently order $n$ for $0 < \alpha<1$ and $2\leq s\leq t$.
In the following, we always assume that $G^*$ is the $A_\alpha$-spectral extremal graph.
In this section, we will prove that $G^*$ has a property of
local edge maximality. Meanwhile, we will apply  the double eigenvectors transformation technique to the $A_{\alpha}(G)$-matrix and prove that
$G^*$  has a property of local degree sequence majorization.

Let $\mathbf{x}=\left(x_{1}, x_{2}, \ldots, x_{n}\right)^{T}$ and $\mathbf{y}=\left(y_{1}, y_{2}, \ldots, y_{n}\right)^{T}$ be two non-increasing real vectors. We say $\mathbf{x}$ is weakly majorized by $\mathbf{y}$, denoted by $\mathbf{x} \prec_{w} \mathbf{y}$, if and only if $\sum_{i=1}^{k} x_{i} \leq \sum_{i=1}^{k} y_{i}$ for  $k=1,2, \ldots, n$. We say  $\mathbf{x}$ is majorized by $\mathbf{y}$, denoted  $\mathbf{x} \prec \mathbf{y}$, if and only if $\mathbf{x} \prec_{w} \mathbf{y}$ and $\sum_{i=1}^{n} x_{i}=\sum_{i=1}^{n} y_{i}$.

\begin{lem}(\cite{Lin})\label{lem::2.8}
Let $\mathbf{x}, \mathbf{y} \in R^{n}$ be two non-negative and non-increasing vectors. If $\mathbf{x} \prec_{w} \mathbf{y}$, then $\|\mathbf{x}\|_{k} \leq\|\mathbf{y}\|_{k}$ for $k>1$, with equality if and only if $\mathbf{x}=\mathbf{y}$.
\end{lem}

The following lemma is from Exercises $5(i)$ on page $74$ of \cite{Zhan}.
\begin{lem}(\cite{Zhan})\label{lem::2.9}
Let $\mathbf{x}, \mathbf{y}, \mathbf{z} \in R^{n}$ be three non-increasing vectors. If $\mathbf{x} \prec \mathbf{y}$, then $\mathbf{x}^{T} \cdot \mathbf{z} \leq \mathbf{y}^{T} \cdot \mathbf{z}$.
\end{lem}

Set for short $\rho_{\alpha} = \rho_{\alpha} (G^{*}) $ and let $\mathbf{X}=(x_{v})_{v \in V(G^{*})} \in R^{n}$ be a Perron eigenvector of $A_{\alpha}(G^{*})$ corresponding to $\rho_{\alpha}$.  By Lemma \ref{lem::2.7},
$G^*$ contains a clique dominating set $K$ of size $s-1$. We immediately get that $G^{*}-K$ has the $(s, t)$-property from Lemma \ref{lem::2.10}. So, $\Delta(G^{*}-K)<t$. For convenience, we let $x_{0}=\sum\limits_{v\in K}x_{v}$. Furthermore, we have $\rho_{\alpha}\geq \rho_{\alpha}\left( K_{s-1} \vee \overline{K}_{n-s+1}\right)\geq \alpha(n-1)+(1-\alpha)(s-2)$ by Lemma \ref{lem::2.1'}. Thus we have the following result.

\begin{lem}\label{lem::3.1}
Let $H$ consist of some  components of $G^{*}-K$ with $|H| \leq N$ (a constant), $x_{1}=\max\limits_{v \in V(H) }x_{v}$ and
$x_{2}=\min\limits_{v \in V(H) }x_{v}$. If $a, b$ are two constants with $a>b.$ Then $a x_{2}>b x_{1}$ and $a x_{2}^{2}>b x_{1}^{2}$.
\end{lem}

\begin{proof}
Since $H\subseteq G^{*}-K$, we have $\Delta(H)\leq\Delta(G^{*}-K) < t$.
For each $v \in V(H)$, we see that
\begin{equation}\label{equ::a}
s-1\leq d_{G^{*}}(v)=d_{H}(v)+s-1 < t+s-1.
\end{equation}
Note that $x_{0}=\sum\limits_{v\in K}x_{v}$. Therefore, we obtain $\rho_{\alpha} x_{1}<(\alpha(t+s-1)+(1-\alpha)t)x_{1}+(1-\alpha)x_{0}$ and $\rho_{\alpha} x_{2} \geq \alpha(s-1)x_{2}+(1-\alpha)x_{0}$, and thus
\begin{equation}\label{equ::4}
x_{1}<\frac{(1-\alpha)x_{0}}{\rho_{\alpha}-(\alpha(t+s-1)+(1-\alpha)t)} \quad \text { and } \quad x_{2} \geq \frac{(1-\alpha)x_{0}}{\rho_{\alpha}-\alpha(s-1)}.
\end{equation}
Since $\rho_{\alpha}\geq \alpha(n-1)+(1-\alpha)(s-2)$, $0<\alpha<1$, $n$ is sufficiently large and $a, b, s, t$ are constants, we can easily get that
$$
ax_{2}-bx_{1}>(1-\alpha)x_{0}(\frac{a}{\rho_{\alpha}-\alpha(s-1)}-\frac{b}{\rho_{\alpha}-(\alpha(t+s-1)+(1-\alpha)t)})>0.
$$
Similarly, we can obtain $a x_{2}^{2}>b x_{1}^{2}$.
\end{proof}

The following lemma implies that $G^{*}$ has a property of local edge maximality.
\begin{lem}\label{lem::3.2}
Let $H$ consist of some  components of $G^{*}-K$ with $|H| \leq N$ (a constant) and $H^{\prime}$ be a graph with $V(H^{\prime})=V(H)$. If $H^{\prime}$ has the $(s, t)$-property, then $e(H^{\prime}) \leq e(H)$.
\end{lem}

\begin{proof}
Suppose to the contrary that $e(H^{\prime})>e(H)$. Let $G^{\prime}=G^{*}-E(H)+E\left( H^{\prime}\right) $. By Lemma \ref{lem::2.10}, $G^{\prime}$ is $K_{s,t}$-minor free since $H^{\prime}$ has the $(s, t)$-property.  Now, by Lemma \ref{lem::3.1} we obtain
$$
\begin{aligned}
\rho_{\alpha}(G^{\prime})-\rho_{\alpha} & \geq \mathbf{X}^{T}(A_{\alpha}(G^{\prime})-A_{\alpha}(G^{*})) \mathbf{X}\\
&=\sum_{u v \in E(H^{\prime})}(\alpha x_{u}^{2}+2(1-\alpha)x_{u}x_{v}+\alpha x_{v}^{2})-\sum_{u v \in E(H)} (\alpha x_{u}^{2}+2(1-\alpha)x_{u}x_{v}+\alpha x_{v}^{2}) \\
& \geq 2 e(H^{\prime}) x_{2}^{2}-2 e(H) x_{1}^{2}>0,
\end{aligned}
$$
which contradicts the maximality of $\rho_{\alpha}$.
\end{proof}

\begin{lem}\label{lem::3.3}
Let $H$ consist of some connected components of $G^{*}-K$ with $|H| \leq N$ (a constant), $H^{\prime}$ be a graph with $V\left(H^{\prime}\right)=V(H)$ and $e\left(H^{\prime}\right)=e\left( H\right)$. If $H^{\prime}$ has the $(s, t)$-property, then
\begin{equation}\label{equ::5}
\sum_{u v \in E\left(H^{\prime}\right)} (d_{H}(u)+ d_{H}(v)) \leq \sum_{u v \in E(H)} (d_{H}(u)+ d_{H}(v)).
\end{equation}
Moreover, if (\ref{equ::5}) holds in equality, then
\begin{equation}\label{equ::6}
\sum_{u v \in E\left(H^{\prime}\right)}(d_{H^{\prime}}(u)+ d_{H^{\prime}}(v)) \leq \sum_{u v \in E(H)}(d_{H}(u)+ d_{H}(v)).
\end{equation}
\end{lem}

\begin{proof}
Let $G^{\prime}=G^{*}-E(H)+E\left(H^{\prime}\right)$ and $\rho_{\alpha}^{\prime}=\rho_{\alpha}(G^{\prime})$.
Clearly, $G^{\prime}$ is also $K_{s, t}$-minor free.
Let $x_{1}=\max\limits_{v \in V(H) }x_{v}$ and $ x_{2}=\min\limits_{v \in V(H) }x_{v}$. Note that $x_{0}=\sum\limits_{v\in K} x_{v}$.
Then by (\ref{equ::a}) and the eigen-equations of $A_{\alpha}(G^{*})$, we have
\begin{equation}\label{equ::8}
\frac{(1-\alpha)(d_{H}(v)x_{2} +x_{0})}{\rho_{\alpha}-\alpha(s-1)}\leq\frac{(1-\alpha)(d_{H}(v)x_{2} +x_{0})}{\rho_{\alpha}-\alpha d_{G^{*}}(v)}\leq x_{v},
\end{equation}
and
\begin{equation}\label{equ::8''}
 x_{v}\leq \frac{(1-\alpha)(d_{H}(v)x_{1} +x_{0})}{\rho_{\alpha}-\alpha d_{G^{*}}(v)}< \frac{(1-\alpha)(d_{H}(v)x_{1} +x_{0})}{\rho_{\alpha}-(\alpha(t+s-1)+(1-\alpha)t)},
\end{equation}
for each vertex $v \in V(H)$.
Let
$a=\sum\limits_{u v \in E(H^{\prime})} (d_{H}(u)+ d_{H}(v))$,
$b=\sum\limits_{u v \in E(H)} (d_{H}(u)+ d_{H}(v))$,
$c_{1}=\sum\limits_{u v \in E(H)} \left( d_{H}^{2}(u)+ d_{H}^{2}(v)\right) $,
$c_{2}=\sum\limits_{u v \in E(H)} d_{H}(u)d_{H}(v)$.
Firstly, we will show that $a\leq b$. Suppose to the contrary that $a\geq b+1$.
By Lemma \ref{lem::3.1}, we obtain $\left(a-\frac{1}{2}\right) x_{2}>b x_{1}$. Moreover, by (\ref{equ::4}) we have
$$x_{0} x_{2}-c_{1} x_{1}^{2}>(1-\alpha)x_{0}^{2}\left(\frac{1}{\rho_{\alpha}-\alpha(s-1)}-\frac{(1-\alpha)c_{1}}{(\rho_{\alpha}-(\alpha(t+s-1)+(1-\alpha)t))^{2}} \right) >0,$$
since $\rho_{\alpha}\geq\alpha(n-1)+(1-\alpha)(s-2)$, $0<\alpha<1$, $n$ is sufficiently large, and  $s, t, c_{1}$ are constants.
Similarly, we can obtain $\frac{1}{2}x_{0} x_{2}-c_{2} x_{1}^{2}>0.$

Now, we denote
$$A_{1}=\sum\limits_{uv\in E(H^{\prime})}( (d_{H}(u)x_{2}+x_{0})^{2}+(d_{H}(v)x_{2}+x_{0})^{2}),\\
A_{2}=\sum\limits_{uv\in E(H)}\left( (d_{H}(u)x_{1} +x_{0})^{2}+(d_{H}(v)x_{1} +x_{0})^{2}\right),$$
$$A_{3}=\sum\limits_{uv\in E(H^{\prime})}(  (d_{H}(u)x_{2}+x_{0})(d_{H}(v)x_{2}+x_{0})) \ \ \mbox{ and} \ \ A_{4}=\sum\limits_{uv\in E(H)}\left( (d_{H}(u)x_{1} +x_{0})(d_{H}(v)x_{1} +x_{0})\right).
$$
By simple scaling, we see that
$A_{1}\geq 2ax_{2}x_{0}+2e(H^{\prime})x_{0}^{2}$, $A_{2}= 2bx_{1}x_{0}+2e(H)x_{0}^{2}+c_{1}x_{1}^{2}$, $A_{3}\geq ax_{2}x_{0}+e(H^{\prime})x_{0}^{2}$ and $A_{4}= bx_{1}x_{0}+e(H)x_{0}^{2}+c_{2}x_{1}^{2}$.
Since $e\left(H^{\prime}\right)=e\left( H\right) $, we get that
\begin{align*}
A_{1}-A_{2}
&\geq (2ax_{2}x_{0}+2e(H^{\prime})x_{0}^{2})-(2bx_{1}x_{0}+2e(H)x_{0}^{2}+c_{1}x_{1}^{2})\notag\\
&=2ax_{2}x_{0}-2bx_{1}x_{0}-c_{1}x_{1}^{2}=2x_{0}\left(\left(a-\frac{1}{2}\right) x_{2}-b x_{1}\right)+(x_{0} x_{2}-c_{1} x_{1}^{2})>0,
\end{align*}
and
\begin{align*}
A_{3}-A_{4}
&\geq (ax_{2}x_{0}+e(H^{\prime})x_{0}^{2})-(bx_{1}x_{0}+e(H)x_{0}^{2}+c_{2}x_{1}^{2})\notag\\
&=ax_{2}x_{0}-bx_{1}x_{0}-c_{2}x_{1}^{2}=x_{0}\left(\left(a-\frac{1}{2}\right) x_{2}-b x_{1}\right)+(\frac{1}{2}x_{0} x_{2}-c_{2} x_{1}^{2})>0,
\end{align*}
which include that $A_{1}>A_{2}$ and $A_{3}>A_{4}$.
Furthermore, let
$$B_{1}= \sum_{uv\in E(H^{\prime})}\left( \frac{(d_{H}(u)x_{2}+x_{0})^{2}+(d_{H}(v)x_{2}+x_{0})^{2}}{(\rho_{\alpha}-\alpha(s-1))^{2}}\right)-\sum_{uv\in E(H)}\left(\frac{(d_{H}(u)x_{1} +x_{0})^{2}+(d_{H}(v)x_{1} +x_{0})^{2}}{(\rho_{\alpha}-(\alpha(t+s-1)+(1-\alpha)t))^{2}}\right),$$
and
$$
B_{2}=\sum_{u v \in E(H^{\prime})}\frac{(d_{H}(u)x_{2}+x_{0})(d_{H}(v)x_{2}+x_{0})}{(\rho_{\alpha}-\alpha(s-1))^{2}}-\sum_{u v \in E(H)} \frac{(d_{H}(u)x_{1} +x_{0})(d_{H}(v)x_{1} +x_{0})}{(\rho_{\alpha}-(\alpha(t+s-1)+(1-\alpha)t))^{2}}.$$
Notice that  $\rho_{\alpha}\geq\alpha(n-1)+(1-\alpha)(s-2)$, $0<\alpha<1$, $n$ is sufficiently large. Then from (\ref{equ::8})-(\ref{equ::8''}), we have
\begin{align*}
\rho_{\alpha}^{\prime}-\rho_{\alpha}\geq& \mathbf{X}^{T}(A_{\alpha}(G^{\prime})-A_{\alpha}(G^{*}))\mathbf{X}\\
=&\sum_{u v \in E(H^{\prime})}(\alpha x_{u}^{2}+2(1-\alpha)x_{u}x_{v}+\alpha x_{v}^{2})-\sum_{u v \in E(H)} (\alpha x_{u}^{2}+2(1-\alpha)x_{u}x_{v}+\alpha x_{v}^{2})\notag\\
=&\alpha\left( \sum_{u v \in E(H^{\prime})}( x_{u}^{2}+ x_{v}^{2})-\sum_{u v \in E(H)} (x_{u}^{2}+x_{v}^{2})\right)+2(1-\alpha)\left(\sum_{u v \in E(H^{\prime})}x_{u}x_{v}-\sum_{u v \in E(H)} x_{u}x_{v} \right) \\
\geq&\alpha(1-\alpha)^{2}B_{1}+2(1-\alpha)^{3}B_{2}\\
=&\alpha(1-\alpha)^{2}\left( \frac{A_{1}}{(\rho_{\alpha}-\alpha(s-1))^{2}}-\frac{A_{2}}{(\rho_{\alpha}-(\alpha(t+s-1)+(1-\alpha)t))^{2}}\right) \\
&+2(1-\alpha)^{3}\left(  \frac{A_{3}}{(\rho_{\alpha}-\alpha(s-1))^{2}}-\frac{A_{4}}{(\rho_{\alpha}-(\alpha(t+s-1)+(1-\alpha)t))^{2}}\right) \\
>&0,\label{equ::6'}
\end{align*}
which contradicts the maximality of $\rho_{\alpha}$.
So (\ref{equ::5}) holds.
	
We next show that the inequality (\ref{equ::6}) holds when $a=b$.
Let $\mathbf{Y}=(y_{v})_{v \in V(G^{*})} \in R^{n}$ be a Perron vector of $A_{\alpha}(G^{\prime})$ corresponding to $\rho_{\alpha}^{\prime}$.
Assume that $y_{0}=\sum\limits_{v \in K} y_{v}$, $y_{1}=\max\limits_{v \in V\left(H^{\prime}\right)} y_{v}$
 and $y_{2}=\min\limits_{v \in V\left(H^{\prime}\right)} y_{v}$. Similar as (\ref{equ::4}), we obtain
\begin{equation}\label{equ::14}
y_{1}<\frac{(1-\alpha)y_{0}}{\rho_{\alpha}^{\prime}-(\alpha(t+s-1)+(1-\alpha)t)} \quad \text { and } \quad y_{2} \geq \frac{(1-\alpha)y_{0}}{\rho^{\prime}_{\alpha}-\alpha(s-1)}.
\end{equation}
Moreover, for any $v\in V(H^{\prime})$  we have
\begin{equation}\label{equ::8'}
\frac{(1-\alpha)(d_{H^{\prime}}(v)y_{2} +y_{0})}{\rho^{\prime}_{\alpha}-\alpha(s-1)}\leq\frac{(1-\alpha)(d_{H^{\prime}}(v)y_{2} +y_{0})}{\rho^{\prime}_{\alpha}-\alpha d_{G^{\prime}}(v)}\leq y_{v},
\end{equation}
and
\begin{equation}\label{equ::8'''}
y_{v}\leq \frac{(1-\alpha)(d_{H^{\prime}}(v)y_{1} +y_{0})}{\rho^{\prime}_{\alpha}-\alpha d_{G^{\prime}}(v)}<\frac{(1-\alpha)(d_{H^{\prime}}(v)y_{1} +y_{0})}{\rho^{\prime}_{\alpha}-(\alpha(t+s-1)+(1-\alpha)t)}.
\end{equation}
Now, we let
\begin{align*}
a^{\prime}&=\sum\limits_{u v \in E(H^{\prime})} (d_{H^{\prime}}(u)+d_{H^{\prime}}(v)),&
b^{\prime}&=\sum\limits_{u v \in E(H)} (d_{H^{\prime}}(u)+d_{H^{\prime}}(v)),\\
c^{\prime}_{1}&=\sum\limits_{u v \in E(H)}(d_{H}(u) d_{H^{\prime}}(v)+d_{H}(v) d_{H^{\prime}}(u)),&
c^{\prime}_{2}&=\sum\limits_{u\in V(H)} d_{H}^{2}(u)d_{H^{\prime}}(u).
\end{align*}
Since $V\left(H^{\prime}\right)=V(H)$, we have
\begin{align*}a&=\sum_{uv \in E\left(H^{\prime}\right)} (d_{H}(u)+d_{H}(v))=\sum_{u \in V\left(H^{\prime}\right)} d_{H}(u) d_{H^{\prime}}(u)=\sum_{u \in V(H)} d_{H}(u) d_{H^{\prime}}(u)\\
&=\sum_{u v \in E(H)} (d_{H^{\prime}}(u)+d_{H^{\prime}}(v))=b^{\prime}.
\end{align*}
Furthermore, we see that
$\sum\limits_{u \in V\left(H^{\prime}\right)} d_{H^{\prime}}^{2}(u)=\sum\limits_{uv \in E\left(H^{\prime}\right)} (d_{H^{\prime}}(u)+d_{H^{\prime}}(v))=a^{\prime}$
and $\sum\limits_{u \in V\left(H\right)} d_{H}^{2}(u)= \sum\limits_{uv \in E\left(H\right)} (d_{H}(u)+d_{H}(v))=b$.
Now we denote
\begin{equation*}
\begin{aligned}
A^{\prime}_{1}=&\sum_{u v \in E(H^{\prime})}((x_{0}+d_{H}(u) x_{2})(y_{0}+d_{H^{\prime}}(v) y_{2})+(x_{0}+d_{H}(v) x_{2})(y_{0}+d_{H^{\prime}}(u) y_{2}))\\
A^{\prime}_{2}=&\sum_{u\in V(H^{\prime})} d_{H^{\prime}}(u) \cdot (d_{H}(u)x_{2} +x_{0}) \cdot(d_{H^{\prime}}(u)y_{2} +y_{0})\\
A^{\prime}_{3}=&\sum_{u v \in E(H)}((x_{0}+d_{H}(u) x_{1})\cdot( y_{0}+d_{H^{\prime}}(v) y_{1})+(x_{0}+d_{H}(v) x_{1})(y_{0}+d_{H^{\prime}}(u) y_{1}))\\
A^{\prime}_{4}=&\sum_{u\in V(H)} d_{H}(u) \cdot (d_{H}(u)x_{1} +x_{0}) \cdot(d_{H^{\prime}}(u)y_{1} +y_{0})
\end{aligned}
\end{equation*}
By simple scaling, we see that
\begin{align*}
A^{\prime}_1\geq&2e(H^{\prime})x_{0}y_{0} + ax_{2}y_{0}+a^{\prime}x_{0}y_{2}, \ \ A^{\prime}_3=2e(H)x_{0}y_{0} + ax_{1}y_{0}+ax_{0}y_{1}+c^{\prime}_{1}x_{1}y_{1},\\
A^{\prime}_2\geq &a^{\prime}x_{0}y_{2}+ax_{2}y_{0} +2e(H^{\prime})x_{0}y_{0},\ \ A^{\prime}_4= c^{\prime}_{2}x_{1}y_{1}+ax_{0}y_{1}+ax_{1}y_{0}+2e(H)x_{0}y_{0}.
\end{align*}
Suppose that (\ref{equ::6}) does not hold, then $a^{\prime} \geq b+1=a+1$.
From Lemma \ref{lem::3.1}, we find that $\left(a+\frac{1}{2}\right) y_{2}>a y_{1}$.
Additionally,  by (\ref{equ::4}) and (\ref{equ::14}) we have
\begin{equation*}\label{equ::22}
\begin{array}{ll}
&\frac{1}{2}x_{0} y_{2}+a y_{0}\left(x_{2}-x_{1}\right)-c^{\prime}_{1} x_{1} y_{1}\\
>&(1-\alpha)x_{0} y_{0}\left(\frac{1}{ 2(\rho^{\prime}_{\alpha}-\alpha(s-1))}+\frac{a}{\rho_{\alpha}-\alpha(s-1)}-\frac{a}{\rho_{\alpha}-(\alpha(t+s-1)+(1-\alpha)t)}-\frac{(1-\alpha)c^{\prime}_{1}}{(\rho_{\alpha}-(\alpha(t+s-1)+(1-\alpha)t))\left(\rho^{\prime}_{\alpha}-(\alpha(t+s-1)+(1-\alpha)t)\right)}\right)\\
>&0,
\end{array}
\end{equation*}
since $\rho_{\alpha}\geq\alpha(n-1)+(1-\alpha)(s-2)$, $\rho^{\prime}_{\alpha}\geq \rho_{\alpha}\left( K_{s-1} \vee \overline{K}_{n-s+1}\right)\geq \alpha(n-1)+(1-\alpha)(s-2)$, $0<\alpha<1$, and $n$ is sufficiently large. Similarly, we have $\frac{1}{2}x_{0} y_{2}+a y_{0}\left(x_{2}-x_{1}\right)-c^{\prime}_{2} x_{1} y_{1}>0,$
Furthermore, since $e\left(H^{\prime}\right)=e\left( H\right)$, we have
\begin{align*}
A^{\prime}_{1}-A^{\prime}_{3}\geq&(2e(H^{\prime})x_{0}y_{0} + ax_{2}y_{0}+a^{\prime}x_{0}y_{2})-(2e(H)x_{0}y_{0} +ax_{1}y_{0}+ax_{0}y_{1}+c^{\prime}_{1}x_{1}y_{1})\\
=& a^{\prime} x_{0} y_{2}+a x_{2} y_{0}-a x_{0} y_{1}-a x_{1} y_{0}-c^{\prime}_{1} x_{1} y_{1}\\
\geq&(a+1) x_{0} y_{2}+a x_{2} y_{0}-a x_{0} y_{1}-a x_{1} y_{0}-c^{\prime}_{1} x_{1} y_{1} \\
=&x_{0}\left(\left(a+\frac{1}{2}\right) y_{2}-a y_{1}\right)+\left(\frac{1}{2}x_{0} y_{2}+a\left(x_{2}-x_{1}\right) y_{0}-c^{\prime}_{1} x_{1} y_{1}\right)\\
>&0,
\end{align*}
and similarly, we obtain $A^{\prime}_{2}-A^{\prime}_{4}>0$.
By (\ref{equ::8})-(\ref{equ::8''}) and (\ref{equ::8'})-(\ref{equ::8'''}), we find that
\begin{align*}
\left(\rho_{\alpha}^{\prime}-\rho_{\alpha}\right) \mathbf{Y}^{T}\mathbf{X} =&\left(A_{\alpha}\left(G^{\prime}\right) \mathbf{Y}\right)^{T}\mathbf{X}-\mathbf{Y}^{T}\left(A_{\alpha}\left(G^{*}\right) \mathbf{X}\right)\\
=&\left( \sum_{u v \in E\left(H^{\prime}\right)}(1-\alpha)\left( x_{u} y_{v}+x_{v} y_{u}\right)+\sum_{u\in V\left(H^{\prime}\right)} \alpha d_{H^{\prime}}(u) x_{u} y_{u}\right) \\
&-( \sum_{u v \in E(H)}(1-\alpha)\left( x_{u}  y_{v}+x_{v}y_{u}\right)+\sum_{u\in V\left(H\right)} \alpha d_{H}(u) x_{u} y_{u})\\
=&(1-\alpha)\left( \sum_{u v \in E\left(H^{\prime}\right)}\left( x_{u} y_{v}+x_{v} y_{u}\right)-\sum_{u v \in E(H)}\left( x_{u}  y_{v}+x_{v}y_{u}\right)\right)\\
&+\alpha\left( \sum_{u\in V\left(H^{\prime}\right)}  d_{H^{\prime}}(u) x_{u} y_{u} - \sum_{u\in V\left(H\right)}d_{H}(u) x_{u} y_{u})\right) \\
\geq&(1-\alpha)^{3}\left( \frac{A^{\prime}_{1}}{(\rho_{\alpha}-\alpha(s-1))(\rho^{\prime}_{\alpha}-\alpha(s-1))}\right.\\
&\left.-\frac{A^{\prime}_{3}}{(\rho_{\alpha}-(\alpha(t+s-1)+(1-\alpha)t))(\rho^{\prime}_{\alpha}-(\alpha(t+s-1)+(1-\alpha)t))}\right) \\
&+\alpha(1-\alpha)^{2}\left( \frac{A^{\prime}_{2}}{(\rho_{\alpha}-\alpha(s-1))(\rho^{\prime}_{\alpha}-\alpha(s-1))}\right.\\
&\left.-\frac{A^{\prime}_{4}}{(\rho_{\alpha}-(\alpha(t+s-1)+(1-\alpha)t))(\rho^{\prime}_{\alpha}-(\alpha(t+s-1)+(1-\alpha)t))}\right) \\
>&0,
\end{align*}
which  contradicts  the maximality of $\rho_{\alpha}$.  Therefore,  (\ref{equ::6}) holds.
\end{proof}

The following lemma implies that $G^{*}$ has a property of local degree sequence majorization.
\begin{lem}\label{lem::3.4}
Let $H$ consist of some components of $G^{*}-K$ with $|H| \leq N$ (a constant), $H^{\prime}$ be a graph with $V\left(H^{\prime}\right)=V(H)$, $e\left(H^{\prime}\right)=e\left( H\right)$ and $H^{\prime}$ have the $(s, t)$-property. If $\pi(H) \prec \pi\left(H^{\prime}\right)$, then $\pi(H)=\pi\left(H^{\prime}\right)$.
\end{lem}

\begin{proof}
Suppose to the contrary that $\pi(H) \neq \pi\left(H^{\prime}\right)$. Since $V(H)=V\left(H^{\prime}\right)$, we have $\pi(H)=\left(d_{1}, d_{2}, \ldots, d_{|H|}\right)$ and $\pi\left(H^{\prime}\right)=\left(d_{1}^{\prime}, d_{2}^{\prime}, \ldots, d_{|H|}^{\prime}\right)$. Let $k=2$ in Lemma \ref{lem::2.8}, we have
\begin{equation}\label{equ::25}
\sum_{i=1}^{|H|} d_{i}^{2}<\sum_{i=1}^{|H|} d_{i}^{\prime 2}.
\end{equation}
Furthermore, let $\mathbf{x}=\mathbf{z}=\pi(H)$ and $\mathbf{y}=\pi\left(H^{\prime}\right)$ in Lemma \ref{lem::2.9}, we obtain
\begin{equation}\label{equ::26}
\sum_{i=1}^{|H|} d_{i}^{2} \leq \sum_{i=1}^{|H|} d_{i} d_{i}^{\prime}.
\end{equation}
Notice that $\sum\limits_{i=1}^{|H|} d_{i}^{2}=\sum\limits_{u v \in E(H)}\left(d_{H}(u)+d_{H}(v)\right)$ and
$$\sum_{i=1}^{|H|} d_{i} d_{i}^{\prime}=\sum_{v \in V\left(H^{\prime}\right)} d_{H}(v) d_{H^{\prime}}(v)=\sum_{u v \in E\left(H^{\prime}\right)}\left(d_{H}(u)+d_{H}(v)\right).$$
Then by (\ref{equ::5})  we have $\sum\limits_{i=1}^{|H|} d_{i} d_{i}^{\prime} \leq \sum\limits_{i=1}^{|H|} d_{i}^{2}$. Together with (\ref{equ::26}), we obtain $\sum\limits_{i=1}^{|H|} d_{i} d_{i}^{\prime}=\sum\limits_{i=1}^{|H|} d_{i}^{2}$. Therefore, (\ref{equ::6})  holds, that is, $\sum\limits_{i=1}^{|H|} d_{i}^{\prime 2}\leq \sum\limits_{i=1}^{|H|} d_{i}^{2}$, which contradicts (\ref{equ::25}).
\end{proof}

\section{The characterization of the $A_\alpha$-spectral extremal graph}
Recall that $G^*$ is a graph with maximum $A_\alpha$-spectral radius among all $K_{s,t}$-minor free graphs of sufficiently large order $n$, where $0 < \alpha<1$ and $2\leq s\leq t$. Moreover, $\rho_{\alpha} = \rho_{\alpha}(G^{*})$, $\mathbf{X}=(x_{v})_{v \in V(G^{*})} \in R^{n}$ is a Perron eigenvector of $A_{\alpha}(G^{*})$ corresponding to $\rho_{\alpha}$, $x_{0}=\sum\limits_{v\in K}x_{v}$. In addition,
$G^*$ contains a clique dominating set $K$ of size $s-1$, $G^{*}-K$ has the $(s, t)$-property and $\Delta(G^{*}-K)<t$.

\begin{lem}\label{lem::3.1'}
If $s=2$, $t=2$, $n-1=pt+r$ and $1\leq r \leq2$. Then $G^{*}\cong K_{1}\vee(pK_{2}\cup K_{r}).$
\end{lem}
\begin{proof}
Note that $|K|=s-1=1$. Let $v$ be the unique vertex in $K$. Thus $d_{G^{*}}(v)=n-1$. Since $G^{*}$ is a $K_{2,2}$-minor free graph, then $G^{*}$ contains no $K_{2,2}$ as a subgraph, which implies that $d_{G^{*}-v}(u)\leq1$ for any vertex $u \in V(G^{*}-v)$. Thus $G^{*}$ is a subgraph of $K_{1}\vee(pK_{2}\cup K_{r})$. By the maximality of  $\rho_{\alpha}$, we get that $G^{*}\cong K_{1}\vee(pK_{2}\cup K_{r})$.
\end{proof}

\begin{lem}\label{lem::3.1''}
If $s=2$, $t=3$, $n-1=pt+r$ and $1\leq r \leq3$. Then $G^{*}\cong K_{1}\vee(pK_{3}\cup K_{r}).$
\end{lem}
\begin{proof}
Note that $|K|=s-1=1$. Let $v$ be the unique vertex in $K$. Thus $d_{G^{*}}(v)=n-1$. Since $G^{*}$ is $K_{2,3}$-minor free, we see that $G^{*}-v$ does not contain any cycle of length at least $4$ as a subgraph
and    $d_{G^{*}-v}(u)\leq2$ for any vertex $u \in V(G^{*}-v)$. Therefore, each component of $G^{*}-v$ is either a triangle or a path of order at least $1$. In fact, there is at most one component of $G^{*}-v$ is a path. Otherwise adding an edge to two pendant vertices in two different components which are paths leads to a $K_{2,3}$-minor free graph with larger $A_\alpha$-spectral radius, a contradiction. Let $P$ be the unique component of $G^{*}-v$ which is not triangle, that is,
$P$ is a path of order $d \geq 1$.  Now, we assert that $1 \leq d \leq 2$.
Let $v_{i}\in V(P)$ such that $P=v_1\ldots v_d$ for $1\leq i \leq d$. We next consider the following three cases.

{\flushleft\bf Case 1. $d=3$.} Let $G^{\prime}=G^{*}+v_1v_d$. Clearly, $G^{\prime}$ is a $K_{2,3}$-minor free graph and $\rho_{\alpha}\left(G^{\prime}\right)>\rho_{\alpha}$ by Lemma \ref{lem::2.2}, a contradiction.

{\flushleft\bf Case 2. $d=4$.} We have $x_{v_1}=x_{v_4}$ and $x_{v_2}=x_{v_3}$ by symmetry. Let $G^{\prime}=G^{*}-v_1 v_2+v_2 v_4$. Clearly, $G^{\prime}$ is $K_{2,3}$-minor free and
$$
\rho_{\alpha}\left(G^{\prime}\right)-\rho_{\alpha} \geq\left(\alpha x^{2}_{v_2}+2(1-\alpha)x_{v_2} x_{v_4}+\alpha x^{2}_{v_4}\right)-\left(\alpha x^{2}_{v_1}+2(1-\alpha)x_{v_1} x_{v_2}+\alpha x^{2}_{v_2}\right)=0.
$$
If $\rho_{\alpha}\left(G^{\prime}\right)=\rho_{\alpha}$, then $\mathbf{X}$ is also an unit eigenvector corresponding to $\rho_{\alpha}\left(G^{\prime}\right)$. Since $v_2$, $v_3$, and $v_4$ are symmetric in $G^{\prime}$, we have $x_{v_2}=x_{v_3}=x_{v_4}$, which implies that $x_{v_1}=x_{v_2}=x_{v_3}=x_{v_4}$. By eigenequations of $A_{\alpha}(G^{*})$ on $v_1$ and $v_2$, we have
$$
\rho_{\alpha} x_{v_1}=2\alpha x_{v_1}+(1-\alpha)(x_{v_2}+x_v) \quad \text { and } \quad \rho_{\alpha} x_{v_2}=3\alpha x_{v_2}+(1-\alpha)(x_{v_1}+x_{v_3}+x_v),
$$
a contradiction. Thus, $\rho_{\alpha}\left(G^{\prime}\right)>\rho_{\alpha}$, which is also a contradiction.

{\flushleft\bf Case 3. $d\geq5$.} We firstly assume that $d$ is odd, let $d=2j+1$ with $j \geq 2$. Then we have $x_{v_k}=x_{v_{2j+2-k}}$ for $1 \leq k \leq j$ by symmetry. Let $G^{\prime}=G^{*}-\left\{v_{j-1} v_j, v_{j+2} v_{j+3}\right\}+\left\{v_j v_{j+2}, v_{j-1} v_{j+3}\right\}$. Clearly, $G^{\prime}$ is a $K_{2,3}$-minor free graph and
$$
\begin{aligned}
\rho_{\alpha}\left(G^{\prime}\right)-\rho_{\alpha}\geq& \sum_{u w \in E(G^{\prime})}(\alpha x_{u}^{2}+2(1-\alpha)x_{u}x_{w}+\alpha x_{w}^{2})-\sum_{u w \in E(G^{*})} (\alpha x_{u}^{2}+2(1-\alpha)x_{u}x_{w}+\alpha x_{w}^{2}) \\
=&2(1-\alpha) (x_{v_j} x_{v_{j+2}}+x_{v_{j-1}} x_{v_{j+3}}-x_{v_{j-1}} x_{v_{j}}-x_{v_{j+2}} x_{v_{j+3}}) \\
=& 2(1-\alpha)( x_{v_{j}}^2+x_{v_{j-1}}^2-2x_{v_{j-1}} x_{v_{j}})  \\
=& 2(1-\alpha)(x_{v_{j-1}}-x_{v_{j}})^2 \\
\geq & 0.
\end{aligned}
$$
Now, suppose that $d$ is even, let $d=2 j$ with $j \geq 3$. Then we have $x_{v_k}=x_{v_{2j+1-k}}$ for $1 \leq k \leq j$ by symmetry. Let $G^{\prime}=G^{*}-\left\{v_{j-1} v_j, v_{j+2} v_{j+3}\right\}+\left\{v_j v_{j+2}, v_{j-1} v_{j+3}\right\}$. Clearly, $G^{\prime}$ is $K_{2,3}$-minor free and
$$
\begin{aligned}
\rho_{\alpha}\left(G^{\prime}\right)-\rho_{\alpha}
\geq& \sum_{u w \in E(G^{\prime})}(\alpha x_{u}^{2}+2(1-\alpha)x_{u}x_{w}+\alpha x_{w}^{2})-\sum_{u w \in E(G^{*})} (\alpha x_{u}^{2}+2(1-\alpha)x_{u}x_{w}+\alpha x_{w}^{2}) \\
=&2(1-\alpha) (x_{v_{j}} x_{v_{j+2}}+ x_{v_{j-1}} x_{v_{j+3}}-x_{v_{j-1}} x_{v_{j}}-x_{v_{j+2}} x_{v_{j+3}})\\
=&2(1-\alpha) ( x_{v_{j}}(x_{v_{j+2}}-x_{v_{j-1}})- x_{v_{j+3}}(x_{v_{j+2}}-x_{v_{j-1}})) \\
=&2(1-\alpha) (x_{v_{j}}-x_{v_{j+3}})(x_{v_{j+2}}-x_{v_{j-1}})\\
=&0.
\end{aligned}
$$
Thus, whether $d$ is odd or even, if $\rho_{\alpha}\left(G^{\prime}\right)=\rho_{\alpha}$, then $\mathbf{X}$ is also an unit eigenvector corresponding to $\rho_{\alpha}\left(G^{\prime}\right)$. Since $v_{v_{j}}$, $ v_{v_{j+1}}$, and $v_{v_{j+2}}$ are symmetric in $G^{\prime}$, we have $x_{v_{j-1}}=x_{v_{j}}=x_{v_{j+1}}$. By the eigenequations of $A_{\alpha}(G^{*})$, we have $x_{v_{1}}=\cdots=x_{v_{d}}$. From the eigenequations of $A_{\alpha}(G^{*})$ on $v_1$ and $v_2$, we have
$$
\rho_{\alpha} x_{v_1}=2\alpha x_{v_1}+(1-\alpha)(x_{v_2}+x_v) \quad \text { and } \rho_{\alpha} x_{v_2}=3 \alpha x_{v_2}+(1-\alpha)(x_{v_1}+x_{v_3}+x_v),
$$
a contradiction. So, $\rho_{\alpha}\left(G^{\prime}\right)>\rho_{\alpha}$, which is also a contradiction.

Therefore, we get that $1 \leq d \leq 2$, and so $G^{*}-v$ consists of disjoint copies of triangles and at most a path of order $1$ or $2$, that is, $G^{*}\cong K_{1}\vee(pK_{2}\cup K_{r}).$
\end{proof}

\begin{lem}\label{lem::3.1'''}
If $s=3$, $t=3$, $n-1=pt+r$ and $1\leq r \leq3$. Then $G^{*}\cong K_{2}\vee(pK_{3}\cup K_{r}).$
\end{lem}
\begin{proof}
Note that $|K|=s-1=2$. Let $K=\{v, v^{\prime}\}$. Thus, $d_{G^{*}}(v)=d_{G^{*}}(v^{\prime})=n-1$. Since $G^{*}$ is $K_{3,3}$-minor free, we see that $G^{*}-\{v, v^{\prime}\}$ does not contain any cycle of length at least $4$ as a subgraph. Moreover, $G^{*}$ contains no $K_{3,3}$ as a subgraph, which implies that $d_{G^{*}-\{v, v^{\prime}\}}(u)\leq2$ for any vertex $u \in V(G^{*}-\{v, v^{\prime}\})$.
Therefore, each component of $G^{*}-\{v, v^{\prime}\}$ is either a triangle or a path of order at least $1$. In fact, there is at most one component of $G^{*}-\{v, v^{\prime}\}$ is a path. Otherwise, adding an edge to two pendant vertices in two different components which are paths leads to a $K_{3,3}$-minor free graph with larger $A_\alpha$-spectral radius, a contradiction. Let $P$ be the unique component of $G^{*}-\{v, v^{\prime}\}$ which is not triangle, that is, $P$ is a path of order $d \geq 1$.
Similar as Cases 1-3 in the proof of Lemma \ref{lem::3.1''}, we can get that $1 \leq d \leq 2$.
Thus $G^{*}-\{v, v^{\prime}\}$ consists of disjoint copies of triangles and at most a path of order $1$ or $2$, and so $G^{*}\cong K_{2}\vee(pK_{3}\cup K_{r}).$
\end{proof}

Lemmas \ref{lem::3.1'}-\ref{lem::3.1'''} imply that the cases $t=2$ and $t=3$ in Theorem \ref{thm::1.1} hold. Thus, it remains to prove the Theorem \ref{thm::1.1} for $t\geq4$. We now let $t\geq4$, $\beta=\left\lfloor\frac{t+1}{s+1}\right\rfloor$, $\gamma = \min \left\{s,\left\lfloor\frac{t+1}{2}\right\rfloor\right\}$. In addition,
we use $H_{i}$, $H_{>i}$ and $H_{<i}$ to denote the family of components in  $G^{*}-K$ with order $i$, greater than $i$, and less than $i$, respectively.

\begin{lem}\label{lem::3.6}
$H_{>t+3}=\varnothing$.
\end{lem}

\begin{proof}
Suppose to the contrary that there exists a component $H \in H_{>t+3}$. Let $|H|=pt+r$, where $p \geq 1$ and $1 \leq r \leq t$. Since $H$ is $K_{1,t}$-minor free, by Lemma \ref{lem::2.4} we have
\begin{equation}\label{equ::27}
e(H) \leq \binom{t}{2}+|H|-t= \binom{t}{2} +(p-1) t+r.
\end{equation}
Now, let $H^{\prime} \cong p K_{t} \cup K_{r}$ with $V\left(H^{\prime}\right)=V(H)$. Clearly, $H^{\prime}$ also has the $(s, t)$-property. Since $|H|>t+3$ and $t \geq 4$, then by (\ref{equ::27}) we obtain
$
e(H)<p \binom{t}{2} + \binom{r}{2} =e\left(H^{\prime}\right)
$,
contradicting Lemma \ref{lem::3.2}.
\end{proof}

\begin{lem}\label{lem::3.10}
There exists at most one component in $H_{<t} \cup H_{t+2}$.
\end{lem}

\begin{proof}
Suppose to the contrary that there exists two components $D_{1}, D_{2} \in H_{<t} \cup H_{t+2}$ with $|D_{1}|\leq|D_{2}|$. For $i=1,2$, we have $e\left(D_{i}\right) \leq\binom{\left|D_{i}\right| }{2}$ if $|D_{i}| < t$, and by
Lemma \ref{lem::2.4} we obtain
$e\left(D_{i}\right) \leq\binom{t }{2}+2$ if $|D_{i}|=t+2$.
Now, let $|D_{1}|+|D_{2}|=pt+r$, where $0 \leq p \leq 2$ and $1 \leq r \leq t$.
Moreover, let $H \cong p K_{t} \cup K_{r}$ with $V\left(H\right)=V(D_{1} \cup D_{2})$. Clearly, $H$ has the $(s, t)$-property.
However, in any case of $|D_{1}|\leq|D_{2}|<t$, $|D_{1}|<t<|D_{2}|$ or $|D_{1}|=|D_{2}|=t+2$, we always find that $e\left(D_{1} \cup D_{2}\right)<p\binom{t}{2}+\binom{r}{2}=e\left(H\right)$, contradicting Lemma \ref{lem::3.2}.
\end{proof}

\begin{lem}\label{lem::3.7}
Let $\eta$ be the number of the components in $H_{t}$. Then $\eta=O\left(\frac{n}{t}\right)$.
\end{lem}

\begin{proof}
From Lemma \ref{lem::3.6}, we have $H_{>t+3}=\varnothing$. Furthermore, there exists at most one component in $H_{<t} \cup H_{t+2}$ by Lemma \ref{lem::3.10}. Note that $t\geq4$. Hence, we just need to show that the number of the components in $H_{i}$ is less than $t$ for each $i\in \{t+1, t+3\}$. Suppose to the contrary that there exists a $H_{i}$ which contains at least $t$ components for $i\in \{t+1, t+3\}$. Let $H^{\prime}_{i}$ be a disjoint union of $t$ components in $H_{i}$.
If $i=t+1$, then by Lemma \ref{lem::2.12} and $\beta=   \left\lfloor\frac{t+1}{s+1}\right\rfloor<\frac{t+1}{2}$, we have
$
e(H^{\prime}_{i}) \leq t\left( \binom{t}{2} +\beta-1\right)<(t+1) \binom{t}{2} =e\left(i K_{t}\right).
$
If $i =t+3$, then by Lemma \ref{lem::2.4}, we obtain
$
e(H^{\prime}_{i}) \leq t\left( \binom{t}{2} +i-t\right)<i\binom{t}{2}=e\left(i K_{t}\right).
$
Clearly, $iK_{t}$ has the $(s, t)$-property. In either of the above cases, we always get a contradiction to Lemma \ref{lem::3.2}.
\end{proof}

\begin{lem}\label{lem::3.5}
If $H$ is a component in $H_{t+1}$. Then $H \cong \overline{H_{s, t}}$, $e\left(\overline{H_{s, t}}\right)=\binom{t}{2}+\beta-1$ and $\beta \geq 2$.
\end{lem}

\begin{proof}
Recall that $\beta=\left\lfloor\frac{t+1}{s+1}\right\rfloor$ and $\gamma = \min \left\{s,\left\lfloor\frac{t+1}{2}\right\rfloor\right\}$. Clearly, $H$ has the $(s, t)$-property.
Moreover, we find that $H$ is an edge-maximal graph by Lemma \ref{lem::3.2}.
Hence, by Lemmas \ref{lem::2.11} and \ref{lem::2.12}, we get that $e(H)=\binom{t}{2}+\beta-1$, $\overline{H}$ is a forest with $\beta$ components and each of which has at least $\gamma+1$ vertices. Suppose that $\overline{H}$ contains a component $D$ which is not a star. Then $D$ contains two leaves $v_{1}$ and $v_{2}$ with distinct neighbors, say $N_{D}(v_{i})=\{u_{i}\}$ for $i=1,2$. Assume without loss of generality that $d_{D}(u_{1}) \leq d_{D}(u_{2})$. Then  $d_{H}(u_{1}) \geq d_{H}(u_{2})$. Let $H^{\prime}=H-\{u_{2} v_{1}\}+\{u_{1} v_{1}\}$. By Lemma \ref{lem::2.11}, $H^{\prime}$ also has the $(s, t)$-property. Furthermore, by Lemma \ref{lem::2.13} we get that $\pi(H) \prec \pi\left(H^{\prime}\right)$ and $\pi(H) \neq \pi\left(H^{\prime}\right)$, which contradicts Lemma \ref{lem::3.4}. Therefore, $\overline{H}$ is a star forest. Notice that $H$ is connected and each component of $\overline{H}$ is a star of order at least $\gamma+1$. Thus, $\overline{H}$ has at least two components, that is, $\beta \geq 2$. Then by $\beta=\left\lfloor\frac{t+1}{s+1}\right\rfloor$ we have $s\leq\left\lfloor\frac{t-1}{2}\right\rfloor$. Note that $\gamma = \min \left\{s,\left\lfloor\frac{t+1}{2}\right\rfloor\right\}$. So, $\gamma=s$. That is, each component of $\overline{H}$ has at least $s+1$ vertices.

Clearly, $|\overline{H_{s,t}}|=t+1$, and we get that $\overline{H_{s, t}}$ has the $(s, t)$-property by by Lemma \ref{lem::2.16}(i).
In order to show that $H \cong \overline{H_{s, t}}$, we just need to show that $\overline{H}$ has at most one component not isomorphic to $K_{1, s}$. Suppose to the contrary that there exists two components $D_{1}$ and $D_{2}$ with $|D_{2}| \geq|D_{1}| \geq s+2$. For $i \in\{1,2\}$, let $v_{i}$ be a leaf and $u_{i}$ be the central vertex of $D_{i}$. Let $H^{\prime}=H-\{u_{2} v_{1}\}+\{u_{1} v_{1}\}$. We get that $\pi(H) \prec \pi\left(H^{\prime}\right)$ and $\pi(H) \neq \pi\left(H^{\prime}\right)$ in a similar way as above, which contradicts Lemma \ref{lem::3.4}. Thus, $H \cong \overline{H_{s, t}}$.
\end{proof}

\begin{lem}\label{lem::3.12}
There are at most $2(\beta-1)$ components in $H_{t+1}$.
\end{lem}

\begin{proof}
The case $H_{t+1}=\varnothing$ is trivial. Assume that $H_{t+1} \neq \varnothing$. Then by Lemma \ref{lem::3.5}, we have $H \cong \overline{H_{s, t}}$, $e(H)=\binom{t}{2}+\beta-1$ and $\beta \geq 2$ for each component $H \in H_{t+1}$. If there exists at least $2\beta$ components in $H_{t+1}$. Then we select $2 \beta$ copies of $\overline{H_{s, t}}$ and denote it by $H^{\prime}_{t+1}$. Now let $H^{\prime\prime}_{t+1} \cong 2 \beta K_{t} \cup K_{2 \beta}$. Since $2 \beta<t$, then $H^{\prime\prime}_{t+1}$ also has the $(s, t)$-property. However, we find that
$$
e\left(H^{\prime\prime}_{t+1}\right)=2 \beta\binom{t}{2}+ \binom{2 \beta}{2} >2 \beta\left(\binom{t}{2}+\beta-1\right)=e(H^{\prime}_{t+1}),
$$
which contradicts Lemma \ref{lem::3.2}. Hence, there exists at most $2\beta-1$ components in $H_{t+1}$.
If there are exactly $2\beta-1$ components in $H_{t+1}$. Then  $H_{t+1}\cong(2 \beta-1)\overline{H_{s, t}}$. Let  $H^{\prime\prime\prime}_{t+1} \cong(2 \beta-1) K_{t} \cup K_{2 \beta-1}$. Thus,  $e\left(H^{\prime\prime\prime}_{t+1}\right)=e(H_{t+1})$. Note that $H_{s, t} \cong (\beta-1) K_{1, s} \cup K_{1, \alpha} $, where $\alpha=t-(s+1)(\beta-1) \geq s$. Thus we obtain
\begin{equation}\label{equ::30}
\delta\left(\overline{H_{s, t}}\right)=(s+1)(\beta-1)>2(\beta-1),
\end{equation}
which implies that $\delta(H_{t+1})>2 \beta-2$. Clearly,  $\pi\left(H^{\prime\prime\prime}_{t+1}\right)=(t-1, \ldots, t-1,2 \beta-2, \ldots, 2 \beta-2)$. Then $\pi(H_{t+1}) \prec \pi\left(H^{\prime\prime\prime}_{t+1}\right)$, which contradicts Lemma \ref{lem::3.4}. Thus, there exists at most $2(\beta-1)$ components in $H_{t+1}$.
\end{proof}

\begin{lem}\label{lem::3.9}
$H_{t+3}=\varnothing$.
\end{lem}

\begin{proof}
Suppose to the contrary that there exists a component $H$ of $G^{*}-K$ with $|H|=t+3$. Then we have $e(H) \leq\binom{t }{2}+3$ by Lemma \ref{lem::2.4}. Clearly, $K_{t} \cup K_{3}$ also has the $(s, t)$-property. Then by Lemma \ref{lem::3.2}, we have $e(H) \geq  e\left(K_{t} \cup K_{3}\right)=\binom{t}{2}+3$. Hence, $e(H)=\binom{t}{2}+3$.
Suppose that $\beta \geq 3$. Since $|\overline{H_{s, t}}|=t+1$, $\overline{H_{s,t}}$ has the $(s, t)$-property by Lemma \ref{lem::2.16}(i) and $e\left(\overline{H_{s, t}}\right)=\binom{t}{2}+\beta-1$ by Lemma \ref{lem::3.5}. Moreover, from Lemma \ref{lem::3.7}, there exists two components which are isomorphic to $K_{t}$ in $G^{*}-K$. Then
$
e\left(H \cup 2K_{t}\right)=3\binom{t}{2}+3<3\binom{t}{2}+3(\beta-1)=e\left(3\overline{H_{s, t}}\right),
$
which contradicts Lemma \ref{lem::3.2}. Hence $\beta \leq 2$.
	
\begin{clm}\label{clm::3.5}
$2 \leq d_{H}(v) \leq t-1$ for each $v \in V(H)$.
\end{clm}
	
\begin{proof}
Clearly, $\Delta(H)\leq\Delta(G^{*}-K)\leq t-1$. Thus, it remains to show that $\delta(H) \geq 2$. Suppose to the contrary that $\delta(H)=1$. Take an arbitrary vertex $v \in V(H)$ with $d_{H}(v)=1$. Then $H-\{v\}$ is connected and  $e(H-v)=\binom{t}{2}+2$.
By Lemmas \ref{lem::2.15} and \ref{lem::3.2},
$H-v$ is isomorphic to either $\overline{H^{\star}}$ or some $\overline{H_{a, b, c}}$ with $a+b+c=t-1$. If $H-v$ is isomorphic to some $\overline{H_{a, b, c}}$, by contracting the edges $u_{1} u_{2}$ and $u_{2} w$ from Fig. \ref{fig-1}, we find that $\overline{H_{a, b, c}}$ contains a $K_{t}$-minor, this implies that $H$ contains a $K_{1, t}$-minor, a contradiction. Hence, $H-v \cong \overline{H^{\star}}$. Now, let $vuw$ be a path of length $2$ in $H$. Note that any two non-adjacent vertices in the Petersen graph $H^{\star}$ have exactly one common neighbor. So, if we contract $uw$ in $H-\{v\}$, then the new vertex is of degree $t-1$ in the resulting graph. Consequently, if we contract $uw$ in $H$, then the new vertex is of degree $t$, which implies that $H$ contains a $K_{1, t}$-minor, a contradiction.
\end{proof}
By $e(H)=\binom{t}{2}+3$ and Claim \ref{clm::3.5}, we have $\pi(H) \prec (t-1, \ldots, t-1,2,2,2)=\pi\left(K_{t} \cup K_{3}\right)$. Then by Lemma \ref{lem::3.4}, we obtain $\pi(H)=\pi\left(K_{t} \cup K_{3}\right)$.
Let $M_{1}=\{v \in V(H) \mid d_{H}(v)=2\}$ and $M_{2}=\cup_{v \in M_{1}} N_{H}(v) \backslash M_{1}$. Clearly, $1 \leq d_{G[M_{1}]}(u) \leq 3$ for any $u \in M_{2}$.
	
\begin{clm}\label{clm::3.6}
$d_{G[M_{1}]}(u)=1$ for any $u \in M_{2}$.
\end{clm}
	
\begin{proof}
Let $L$ be the set of non-adjacent vertex-pairs in $M_{2}$. Since $|V(H) \backslash M_{1}|=t$ and $d_{H}(u)=t-1$ for each $u \in V(H) \backslash M_{1}$, we have $|L|=\frac{1}{2} e(M_{1}, M_{2})$, and so $1 \leq|L| \leq 3$. Suppose that $d_{G[M_{1}]}\left(u_{0}\right)=c \in\{2,3\}$ for some $u_{0} \in M_{2}$. Then $u_{0}$ has exactly $c$ non-neighbors, say $\left\{u_{1}, \ldots, u_{c}\right\}$, in $M_{2}$. If $c=3$, then $L=\left\{\left(u_{0}, u_{i}\right) \mid i=1,2,3\right\}$  and there are three paths $u_{0} v_{i} u_{i}$ in $H$, where $v_{i} \in M_{1}$ and $i \in\{1,2,3\}$. Now, if we contract two of the three paths into edges, then the resulting graph is isomorphic to $S^{1}(K_{t})$. However, $S^{1}\left(K_{t}\right)$ contains $K_{2, t-1}$ as a subgraph, which contradicts the fact that $H$ has the $(s,t)$-property. Therefore, $c=2$. Since $t \geq 4$, we can find a vertex $u_{3} \in N_{H}(u_{0}) \backslash M_{1}$. Note that $|L| \leq 3$. Hence, $\{u_{1}, u_{2}\} \cap N_{H}(u_{3}) \neq \varnothing$. Moreover, if $u_{1}, u_{2} \in N_{H}\left(u_{3}\right)$, then $H$ contains a double star with a non-pendant edge $u_{0} u_{3}$ and $t$ leaves, which implies that $H$ has a $K_{1, t}$-minor, a contradiction. So we may assume that  $u_{1} u_{3} \in E(H)$ and $u_{2} u_{3} \notin E(H)$. Since $|L| \leq 3$, we have $u_{1} u_{2} \in E(H)$. It follows that $P=u_{0} u_{3} u_{1} u_{2}$ is an induced path in $H$. Now, $|L|=3$ and $M_{2}=\left\{u_{0}, u_{1}, u_{2}, u_{3}\right\}$. Furthermore, $d_{G[M_{1}]}\left(u_{i}\right)=2$ for $i \in\{0,2\}$ and  $d_{G[M_{1}]}(u_{i})=1$ for $i \in\{1,3\}$. Then there exists a double star with a non-pendant edge in $E(P)$ and $t$ leaves, which implies that $H$ has a $K_{1, t}$-minor, a contradiction.
\end{proof}
	
\begin{clm}\label{clm::3.7}
$x_{u}>x_{v}$ for any two vertices $u, v$ with $u \in M_{2}$ and $v \in M_{1}$.
\end{clm}
\begin{proof}
Let  $x_{1}=\max\limits_{v \in V(H)}x_{v}$ and $x_{2}=\min\limits_{v\in V(H)}x_{v}$.   Note that $x_{0}=\sum\limits_{v\in K}x_{v}$.
By (\ref{equ::8}) we have $$x_{u}\geq\frac{(1-\alpha)(d_{H}(u)x_{2} +x_{0})}{\rho_{\alpha}-\alpha(s-1)}\ \ \mbox{and} \ \  x_{v}< \frac{(1-\alpha)(d_{H}(v)x_{1} +x_{0})}{\rho_{\alpha}-(\alpha(t+s-1)+(1-\alpha)t)}.$$
Since $d_{H}(u)>d_{H}(v)$, by Lemma \ref{lem::3.1} we have $d_{H}(u)x_{2}>d_{H}(v)x_{1}$.
Now we see that
$$
x_{u}-x_{v}> (1-\alpha)\left( \frac{d_{H}(u)x_{2} +x_{0}}{\rho_{\alpha}-\alpha(s-1)}-\frac{d_{H}(v)x_{1} +x_{0}}{\rho_{\alpha}-(\alpha(t+s-1)+(1-\alpha)t)}\right)  >0,
$$
since $\rho_{\alpha}\geq\alpha(n-1)+(1-\alpha)(s-2)$, $0<\alpha<1$, $n$ is sufficiently large, and $s, t$ are constants.
\end{proof}
	
By Claim \ref{clm::3.6}, each $u_{i}\in M_{2}$ has a unique neighbor $v_{i} \in M_{1}$. Thus, each $u_{i}\in M_{2}$ has a unique non-neighbor $u_{j} \in V(H) \backslash M_{1}$, where $u_{j} \in M_{2}$.
If $v_{i}=v_{j}$ for some $(u_{i}, u_{j}) \in L$, then $u_{i}$ and $u_{j}$ have $t-1$ common neighbors in $H$, which implies that $H$ contains a copy of $K_{2, t-1}$, a contradiction.
Hence, $v_{i} \neq v_{j}$ for each $\left(u_{i}, u_{j}\right) \in L$. If $v_{i} v_{j} \in E(H)$ for some $\left(u_{i}, u_{j}\right) \in L$,
then we can also obtain a copy of $K_{2, t-1}$ by contracting the edge $v_{i} v_{j}$, a contradiction. So, $v_{i} v_{j} \notin E(H)$ for each $\left(u_{i}, u_{j}\right) \in L$. Now, Let
$$
H^{\prime}=H-\left\{u_{i} v_{i}, u_{j} v_{j} \mid\left(u_{i}, u_{j}\right) \in L\right\}+\left\{u_{i} u_{j}, v_{i} v_{j} \mid\left(u_{i}, u_{j}\right) \in L\right\},
$$
and $G^{\prime}=G^{*}-E(H)+E\left(H^{\prime}\right)$.
Clearly, $H^{\prime} \cong K_{t} \cup K_{3}$, and so $H^{\prime}$ has the $(s,t)$-property. Furthermore, by Claim \ref{clm::3.7} we have
$$
\begin{aligned}
\rho_{\alpha}(G^{\prime})-\rho_{\alpha} \geq& \sum_{\left(u_{i}, u_{j}\right) \in L} ((\alpha x^{2}_{u_{i}} +2(1-\alpha)x_{u_{i}}x_{u_{j}} +\alpha x^{2}_{u_{j}})+(\alpha x^{2}_{v_{i}}+2(1-\alpha)x_{v_{i}}x_{v_{j}} +\alpha x^{2}_{v_{j}})\\
&-(\alpha x^{2}_{u_{i}}+2(1-\alpha)x_{u_{i}}x_{v_{i}} +\alpha x^{2}_{v_{i}})-(\alpha x^{2}_{u_{j}} +2(1-\alpha)x_{u_{j}}x_{v_{j}}+ \alpha x^{2}_{v_{j}})) \\
=&\sum_{\left(u_{i}, u_{j}\right) \in L} 2(1-\alpha)(x_{u_{i}}-x_{v_{j}})(x_{u_{j}}-x_{v_{i}})>0,
\end{aligned}
$$
a contradiction.
\end{proof}

\begin{lem}\label{lem::3.8}
If $H$ is a component in $H_{t+2}$. Then $\beta \leq 2$, moreover, $H \cong S^{1}\left(\overline{H_{s, t}}\right)$ for $\beta=2$ and $H \cong \overline{H^{\star}}$ for $\beta=1$.
\end{lem}

\begin{proof}
The proof is divided into several claims.
\begin{clm}\label{clm::3.1}
$
e(H)=\binom{t}{2}+2.
$
\end{clm}
\begin{proof}
Since $|H|=t+2$, by Lemma \ref{lem::2.4} we have $e(H) \leq\binom{t }{2}+2$. Clearly, $K_{t} \cup K_{2}$ has the $(s, t)$-property, then by Lemma \ref{lem::3.2} we obtain $e(H) \geq e\left(K_{t} \cup K_{2}\right)=\binom{t }{2}+1$. Suppose that  $e(H)=\binom{t }{2}+1$. Notice that $\Delta(H) \leq t-1$ and $\delta(H) \geq 1$. Then $\pi(H) \prec \pi\left(K_{t} \cup K_{2}\right)$. By Lemma \ref{lem::3.4}, we obtain
$\pi(H)=\pi\left(K_{t} \cup K_{2}\right)=(t-1, t-1, \ldots, t-1,1,1)$.
This implies that $H$ is obtained from $K_{t}$ by deleting an edge $u_{1} u_{2}$ and adding two pendant edges $u_{1} v_{1}$ and $u_{2} v_{2}$. Since $t \geq 4$, we have  $N_{H}\left(u_{1}\right) \backslash\left\{v_{1}\right\} \neq \varnothing$ and
$$
\left( \rho_{\alpha}-\alpha(s+1)+1\right) \left(x_{u_{1}}-x_{v_{1}}\right)=\alpha(t-2)x_{u_{1}}+(1-\alpha)\sum_{v \in N_{H}\left(u_{1}\right) \backslash\left\{v_{1}\right\}} x_{v}>0,
$$
thus $x_{u_{1}}>x_{v_{1}}$. By symmetry, we have $x_{u_{1}}=x_{u_{2}}$ and  $x_{v_{1}}=x_{v_{2}}$. Now, let $H^{\prime}= H-\left\{u_{1} v_{1}, u_{2} v_{2}\right\}+\left\{u_{1} u_{2}, v_{1} v_{2}\right\}$ and $G^{\prime}=G^{*}-E(H)+E\left(H^{\prime}\right)$. Clearly, $H^{\prime} \cong K_{t} \cup K_{2}$, thus $H^{\prime}$ has the $(s, t)$-property. However, we see that
$$
\rho_{\alpha}\left(G^{\prime}\right)-\rho_{\alpha} \geq X^{T}\left(A_{\alpha}\left(G^{\prime}\right)-A_{\alpha}\left(G^{*}\right)\right) X=2(1-\alpha)\left(x_{u_{1}}-x_{v_{2}}\right)\left(x_{u_{2}}-x_{v_{1}}\right)>0,
$$
a contradiction. Therefore, $e(H)=\binom{t}{2}+2$.
\end{proof}

\begin{clm}\label{clm::3.2}
$\beta \leq 2$.
\end{clm}
\begin{proof}
Suppose to the contrary that $\beta \geq 3$.
Then we have $e\left(\overline{H_{s, t}}\right)= \binom{t}{2}+\beta-1 \geq  \binom{t}{2}+2$ by Lemma \ref{lem::3.5}. From Lemma \ref{lem::3.7}, there exists a component which is isomorphic to $K_{t}$ in $G^{*}-K$.
Let $H^{\prime} \cong 2\overline{H_{s, t}}$ and $V\left(H^{\prime}\right)=V\left(H \cup K_{t}\right)$. By Lemma \ref{lem::2.16}(i), $H^{\prime}$ has the $(s, t)$-property. Then by Claim \ref{clm::3.1}, we have $e(H^{\prime})\geq  2\left( \binom{t}{2}+2\right) >2 \binom{t}{2}+2=e({H} \cup K_{t})$, which contradicts
Lemma \ref{lem::3.2}.
\end{proof}

\begin{clm}\label{clm::3.3}
If $\beta=2$, then $H\cong S^{1}\left(\overline{H_{s, t}}\right)$.
\end{clm}
\begin{proof}
Clearly, $|S^{1}\left(\overline{H_{s, t}}\right)|=t+2$. Moreover, Lemma \ref{lem::2.16}(i) gives that $S^{1}\left(\overline{H_{s, t}}\right)$ has the $(s, t)$-property.
Since $\beta=\left\lfloor\frac{t+1}{s+1}\right\rfloor=2$, we have $2s+1 \leq t \leq 3s+1$.
Thus $\gamma=\min \left\{s,\left\lfloor\frac{t+1}{2}\right\rfloor\right\}=s$.
Since $|H|=t+2$, then together with  Claims \ref{clm::3.1}, \ref{clm::3.2} and Lemma \ref{lem::2.15}, we find that $H$ is isomorphic to either $\overline{H^{\star}}$ or some $\overline{H_{a, b, c}}$, where $a+b+c=t-1$.
If $H \cong \overline{H^{\star}}$. Then $|H|=t+2=10$ and $H$ is $6$-regular. Thus, we obtain $t=8$ and $e(H)=30$. Moreover, since $S^{1}\left(\overline{H_{s, 8}}\right)$ is a subdivision of $\overline{H_{s, 8}}$, by Lemma \ref{lem::3.5} we have $e\left(S^{1}\left(\overline{H_{s, 8}}\right)\right)=e\left(\overline{H_{s, 8}}\right)+1
=\binom{t}{2}+\beta=30$. From Lemma \ref{lem::2.16}(i), we get that
$$
\pi\left(S^{1}\left(\overline{H_{s, t}}\right)\right)=(t-1, \ldots, t-1, t-s, s+1,2).
$$
Since $t=8$ and $s\geq 2$, we have $t-s \leq 6$. Hence, $\pi(H) \prec \pi\left(S^{1}\left(\overline{H_{s, 8}}\right)\right)$, which contradicts Lemma \ref{lem::3.4}. Therefore, $H$ is isomorphic to some $\overline{H_{a, b, c}}$. Now, we assert that
\begin{equation}\label{equ::29}
\min \{b, c\} \geq \gamma.
\end{equation}
If $b \leq \gamma-1$, we contract the edge $u_{2} w$ in $\overline{H_{a, b, c}}$ and call the new vertex $u$ in the resulting graph, then we get a complete bipartite subgraph with bipartite partition
$$\left(V(K_{b}) \cup\{u\}, V(K_{a})\cup V(K_{c}) \cup\{u_{1}\}\right).$$
This implies that $H$ contains a $K_{b+1, a+c+1}$-minor, contradicting the $(s, t)$-property. So, $b \geq \gamma$. By symmetry, we have $c \geq \gamma$.

Furthermore, by Fig. \ref{fig-1} we can see that
$$
\pi\left(\overline{H_{a, b, c}}\right)=\left(t-1, \ldots, t-1, a_{1}, a_{2}, a_{3}\right),
$$
where $a_{1}, a_{2}, a_{3} \in\{a+2, b+1, c+1\}$. By (\ref{equ::29})  and $\gamma=s$, we have $\min \{b, c\} \geq \gamma=s$. Thus, we find that $a_{3} \geq 2$ and $a_{2} \geq s+1$. Hence, we see that $\pi\left(\overline{H_{a, b, c}}\right) \prec \pi\left(S^{1}\left(\overline{H_{s, t}}\right)\right)$. Then we obtain $\pi\left(\overline{H_{a, b, c}}\right)=\pi\left(S^{1}\left(\overline{H_{s, t}}\right)\right)$ by Lemma  \ref{lem::3.4},. This implies that $a_{3}=2$ and $a_{2}=s+1$. Note that $\min \{b, c\} \geq s \geq 2$. It follows that $a=0$ and $\min \{b, c\}=s$, that is, $\overline{H_{a, b, c}} \cong S^{1}\left(\overline{H_{s, t}}\right)$.
\end{proof}

\begin{clm}\label{clm::3.4}
If $\beta=1$, then $H \cong \overline{H^{\star}}$.
\end{clm}

\begin{proof}
Since $\beta=\left\lfloor\frac{t+1}{s+1}\right\rfloor=1$, we have $t \leq 2s$. Thus, $\gamma=\min \left\{s,\left\lfloor\frac{t+1}{2}\right\rfloor\right\}=\left\lfloor\frac{t+1}{2}\right\rfloor$. Since $|H|=t+2$, then together with  Claims \ref{clm::3.1}, \ref{clm::3.2} and Lemma \ref{lem::2.15}, $H$ is isomorphic to either $\overline{H^{\star}}$ or some $\overline{H_{a, b, c}}$, where $a+b+c=t-1$. If $H$ is isomorphic to some $\overline{H_{a, b, c}}$, then we get (\ref{equ::29}) in a similar way as above. It follows that $t-1\geq b+c \geq 2 \gamma=2\left\lfloor\frac{t+1}{2}\right\rfloor$, a contradiction.  Hence, $H$ is only possibly isomorphic to $\overline{H^{\star}}$. Since $|\overline{H^{\star}}|=t+2=10$, we have $t=8$. Then by Lemma \ref{lem::2.16}(ii), $\overline{H^{\star}}$ has the $(s, t)$-property. Therefore, $H \cong \overline{H^{\star}}$, as desired.
\end{proof}
Combining Claims \ref{clm::3.2}-\ref{clm::3.4}, we obtain the result of Lemma \ref{lem::3.8}.
\end{proof}

\begin{lem}\label{lem::3.13}
If $H_{t+1} \neq \varnothing$. Then $H_{t+2} \cup H_{<t}=\varnothing$.
\end{lem}

\begin{proof}
Suppose to the contrary that $H_{t+2} \cup H_{<t} \neq \varnothing$. By Lemma \ref{lem::3.10}, $G^{*}-K$ contains a unique component $D_{1} \in H_{t+2} \cup H_{<t}$. Since $H_{t+1} \neq \varnothing$, we take a component $D_{2} \in H_{t+1}$. Then by Lemma \ref{lem::3.5}, we have $\beta \geq 2$ and $D_{2} \cong \overline{H_{s, t}}$.
	
If $|D_{1}|=t+2$, then $\beta=2$ and $D_{1} \cong S^{1}\left(\overline{H_{s, t}}\right)$ by Lemma \ref{lem::3.8}. Note that $D_{1}$ is a subdivision of $\overline{H_{s, t}}$. From Lemma \ref{lem::3.5}, we obtain $e\left(\overline{H_{s, t}}\right)=\binom{t}{2}+\beta-1$. Thus, we have
$$
e\left(D_{1} \cup D_{2}\right)=2 e\left(\overline{H_{s, t}}\right)+1=2\left(\binom{t}{2}+\beta-1\right)+1=2\binom{t}{2}+3.
$$
Now let $H^{\prime}=2 K_{t} \cup K_{3}$. Then $|H^{\prime}|=|D_{1} \cup D_{2}|$ and $e\left(H^{\prime}\right)=e\left(D_{1} \cup D_{2}\right)$. By (\ref{equ::30}), we have $\delta\left(D_{2}\right)>2$. Since $D_{1}$ is a subdivision of $D_{2}$, $\delta\left(D_{1}\right)=2$ and its vertex of degree two is unique. Which implies that $\pi\left(D_{1} \cup D_{2}\right) \prec \pi\left(H^{\prime}\right)$ and  $\pi\left(D_{1} \cup D_{2}\right) \neq \pi\left(H^{\prime}\right)$, this contradicts Lemma \ref{lem::3.4}.
	
If $|D_{1}|<t$, then $D_{1} \cong K_{|D_{1}|}$. Indeed, we have
\begin{equation}\label{equ::31}
|D_{1}| \leq \beta-1.
\end{equation}
Otherwise, we have $|D_{1}|>\beta-1$, then
$$
e\left(D_{1} \cup D_{2}\right)=\binom{\left|D_{1}\right| }{2}+\binom{t}{2}+\beta-1<\binom{t }{2}+\binom{\left|D_{1}\right| }{2}+\left|D_{1}\right|=e\left(K_{t} \cup K_{\left|D_{1}\right|+1}\right),
$$
which contradicts Lemma \ref{lem::3.2}. On the other hand,
we get that $G^{*}-K$ contains the disjoint union of a copy of $D_{1}$ and $|D_{1}|$ copies of $K_{t}$ by Lemma \ref{lem::3.7}. We denote it by $H^{\prime\prime}$. Then $e(H^{\prime\prime})=  \binom{\left|D_{1}\right| }{2} +\left|D_{1}\right|\binom{t}{2}$. Now let $H^{\prime\prime\prime}$ be the disjoint union of $|D_{1}|$ copies of $D_{2}$. Clearly, $|H^{\prime\prime}|=|H^{\prime\prime\prime}|=|D_{1}|(t+1)$ and $e\left(H^{\prime\prime\prime}\right)= |D_{1}|\left(\binom{t}{2}+\beta-1\right)$. By Lemma \ref{lem::3.2}, $e\left(H^{\prime\prime\prime}\right) \leq e(H^{\prime\prime})$.
Therefore, we get that
\begin{equation}\label{equ::32}
\begin{array}{ll}
|D_{1}|\geq 2\beta-1,
\end{array}
\end{equation}
which contradicts (\ref{equ::31}).
\end{proof}

\begin{lem}\label{lem::3.11}
Let $t\geq4$, $\beta=1$, $n-s+1=pt+r$ and $1\leq r\leq t$. Then
$$
G^{*}-K \cong\left\{\begin{array}{ll}
(p-1) K_{t} \cup \overline{H^{\star}} & \text { for } r=2 \text { and } t=8; \\
p K_{t} \cup K_{r} & \text { otherwise. }
\end{array}\right.
$$
\end{lem}

\begin{proof}
By Lemmas \ref{lem::3.6} and \ref{lem::3.9}, $H_{>t+2}=\varnothing$. Note that $\beta=1$. Then by Lemma \ref{lem::3.5}, we also have $H_{t+1}=\varnothing$. Since any component of $G^{*}- K$ of order $t$ is isomorphic to $K_{t}$, then by Lemma \ref{lem::3.10}, there exists at most one component $H$ of $G^{*}-K$ are not isomorphic to $K_{t}$. Notice that  $|G^{*}-K|=pt+r$, where $1 \leq r \leq t$, we see that either $|H|=r$ or $|H|=t+r=t+2$. If $r \neq 2$, then $H \cong K_{r}$ and $G^{*}-K \cong p K_{t} \cup K_{r}$, as desired. Now assume that $r=2$. Then either $H \cong K_{2}$ or $H \cong \overline{H^{\star}}$ by Lemma \ref{lem::3.8}. If $H \cong \overline{H^{\star}}$, then $|H|=t+2=10$, and so $t=8$. Hence, if $t \neq 8$, then $G^{*}-K \cong p K_{t} \cup K_{2}$. Now we consider the case $t=8$. Suppose that $H \cong K_{2}$, then $e\left(K_{8} \cup K_{2}\right)=29<30=e\left(\overline{H^{\star}}\right)$, contradicting Lemma \ref{lem::3.2}. That is, $H\cong \overline{H^{\star}}$.
\end{proof}

\begin{lem}\label{lem::3.14}
Let $t\geq4$, $\beta \geq 2$, $n-s+1=pt+r$ and $1\leq r\leq t$. Then
$$
G^{*}-K \cong\left\{\begin{array}{ll}
(p-1) K_{t} \cup S^{1}\left(\overline{H_{s, t}}\right) & \text { for } r=\beta=2; \\
(p-r) K_{t} \cup r \overline{H_{s, t}} & \text { for } r \leq 2(\beta-1) \text { except } r=\beta=2; \\
p K_{t} \cup K_{r} & \text { for } r>2(\beta-1).
\end{array}\right.
$$
\end{lem}

\begin{proof}
Note that $|G^{*}-K|=pt+r$, where $1 \leq r \leq t$, and $H \cong \overline{H_{s, t}}$ for every $H \in H_{t+1}$.
Now, we assert that if $r \leq 2(\beta-1)$ then $H_{<t}=\varnothing$. Indeed, if $G^{*}-K$ contains a component $D$ with $|D|<t$, then $H_{>t}=\varnothing$ and there exists exactly one component in $H_{<t}$ by Lemmas \ref{lem::3.10} and \ref{lem::3.13}. This implies that $|D|=r$.
From (\ref{equ::32}),  $r=|D| \geq 2\beta-1$, a contradiction.

We firstly assume that $r\neq 2$. Then by Lemmas \ref{lem::3.10} and \ref{lem::3.13}, $G^{*}-K$ is isomorphic to either $p K_{t} \cup K_{r}$ or  $(p-r) K_{t} \cup r \overline{H_{s, t}}$. If $r \leq 2(\beta-1)$, then $r<t$, and  $H_{<t}=\varnothing$. Thus $G^{*}-K \cong(p-r) K_{t} \cup r \overline{H_{s, t}}$. If $r>2(\beta-1)$, then by Lemma  \ref{lem::3.12}, $G^{*}-K \cong pK_{t} \cup K_{r}$.

Next, we assume that $r=2$. Since $\beta \geq 2$, we have $r\leq 2(\beta-1)$ and thus  $H_{<t}=\varnothing$. If $\beta>2$, then $H_{t+2}=\varnothing$ by Lemma \ref{lem::3.8}. Thus, $G^{*}-K \cong(p-r) K_{t} \cup r\overline{H_{s, t}}$. It remains to prove the case $r=\beta=2$. Now if $H_{t+1} \neq \varnothing$, then $H_{t+2} \cup H_{<t}=\varnothing$ by Lemma \ref{lem::3.13}. Moreover, by Lemma \ref{lem::3.12}, the number of the components in $H_{t+1}$ is at most $2(\beta-1)=2$. This implies that  $G^{*}-K \cong(p-2) K_{t} \cup 2\overline{H_{s, t}}$. Notice that $e\left(\overline{H_{s, t}}\right)=\binom{t}{2}+\beta-1=\binom{t}{2}+1$. Now let $H^{\prime} \cong K_{t} \cup S^{1}\left(\overline{H_{s, t}}\right)$. Then  $|H^{\prime}|=2|\overline{H_{s, t}}|=2t+2$ and
$
e\left(H^{\prime}\right)=e\left(K_{t}\right)+e\left(\overline{H_{s, t}}\right)+1=2 e\left(\overline{H_{s, t}}\right).
$
Since $S^{1}\left(\overline{H_{s, t}}\right)$ is a subdivision of $\overline{H_{s, t}}$, and by (\ref{equ::30}) we have $\delta\left(\overline{H_{s, t}}\right)>2$, we can easily see that $\pi\left(2\overline{H_{s, t}}\right) \prec \pi\left(H^{\prime}\right)$, this contradicts Lemma \ref{lem::3.4}. Thus, $H_{t+1}=\varnothing$. It follows that there exists exactly one component in $H_{t+2}$. By Lemma \ref{lem::3.8}, we have $G^{*}-K \cong(p-1) K_{t} \cup S^{1}\left(\overline{H_{s, t}}\right)$.
\end{proof}

Combining Lemmas \ref{lem::3.1'}-\ref{lem::3.1'''}, \ref{lem::3.11} and \ref{lem::3.14}, we obtain the result of Theorem \ref{thm::1.1}.

\end{document}